\documentclass[12pt,a4]{article}

\usepackage{amsmath,amsfonts,amssymb, mathrsfs,epsfig,amsthm}

\topmargin= - 10 true mm
\newcommand{\lag}{\langle}  
\newcommand{\rag}{\rangle}  
\newcommand{\bla}{\mbox{\boldmath$\lambda$}}     
\newcommand{\bx}{\mbox{\boldmath$x$}}        
\newcommand{\bxi}{\mbox{\boldmath$\xi$}}         
\newcommand{\bal}{\mbox{\boldmath$\alpha$}}   
\newcommand{\sbla}{\mbox{\scriptsize\boldmath$\lambda$}}   
\newcommand{\sbal}{\mbox{\scriptsize\boldmath$\alpha$}}   
\newcommand{\bv}[1]{\mbox{\boldmath$#1$}}   
\newcommand{\sbv}[1]{\mbox{\scriptsize\boldmath$#1$}}  
\newcommand{\SU}[1]{\bv{S}(#1)} 
\newcommand{\floor}[1]{\lfloor #1 \rfloor} 
\newcommand{\bbe}{\mbox{\boldmath$\beta$}}     
\newcommand{\sbbe}{\mbox{\scriptsize\boldmath$\beta$}}   

\newtheorem{remark}{Remark}
\newtheorem{theorem}{Theorem}
\newtheorem{lemma}{Lemma}
\newtheorem{corollary}{Corollary}

\DeclareMathOperator*{\argminA}{arg\,min}

\begin{document}

\title{The exact asymptotics of the large deviation\\ probabilities in the multivariate\\ boundary crossing problem}

\author{Yuqing~Pan$^1$ and Konstantin~Borovkov$^2$}

\date{}
\maketitle

\footnotetext[1]{School of Mathematics and Statistics, The University of Melbourne, Parkville 3010, Australia. E-mail: yuqingpan13@gmail.com.}

\footnotetext[2]{School of Mathematics and Statistics, The University of Melbourne, Parkville 3010, Australia. E-mail: borovkov@unimelb.edu.au.}

\footnotetext[3]{Accepted for publication by the Applied Probability Trust in {\em Advances in Applied Probability\/} 51.3 (September 2019), see http://www.appliedprobability.org.}
 




\begin{abstract}
For a multivariate random walk with i.i.d.\ jumps satisfying the Cram\'er moment condition and having mean vector with at least one negative component, we  derive the exact asymptotics of the probability of ever hitting the positive orthant that is being translated to infinity along a fixed vector with positive components. This problem is motivated by and extends results from a paper by F.~Avram et~al.\ (2008) on a two-dimensional risk process. Our approach combines the large deviation techniques  from a series of papers by A.~Borovkov and A.~Mogulskii from around 2000 with new auxiliary constructions, which enable  us to  extend their results on hitting remote sets with smooth boundaries to the case of boundaries with a ``corner" at the ``most probable  hitting point". We also discuss how our results can be extended to the case of more general target sets.


\end{abstract}

\section{Introduction} 


The  present  work was motivated by the following two-dimensional risk model from~\cite{pal1}. Consider two   insurance companies  that divide between them both claims and premia in
specified fixed proportions, so that their risk processes $U_1$ and $U_2$ are, respectively,
\begin{align*}
\label{orig_model}
U_i(t):= u_i + c_i t- S_i(t),\qquad i=1,2,
\end{align*}
where $u_i>0$ are the initial reserves, $c_i>0$ the premium rates, and their respective claim processes  $S_i(t)=\delta_i S (t)$ are  fixed proportions   ($\delta_i>0$ are constants,  $\delta_1+\delta_2=1$) of a common process $S(t)$ of claims made against them.  It is assumed~\cite{pal1}  for definiteness  that $c_1/\delta_1 > c_2/\delta_2$, i.e.,   the second company receives less premium per amount paid out and so can be considered as a  reinsurer.  That paper mostly dealt with the following two ruin times:
\begin{align*}
\tau_{\rm or} := \inf\{t\ge 0: U_1(t) \wedge U_2 (t) <0\},
 \quad
\tau_{\rm sim}    := \inf\{t\ge 0: U_1(t) \vee U_2 (t) <0\},
\end{align*}
at which at least one of the two or both of the companies are ruined, respectively. The key observation made in~\cite{pal1} was  that both times are actually the first crossing times of some piece-wise linear boundaries by the univariate claim process $S(t)$. Thus the problem of computing the respective ``bivariate ultimate ruin probabilities"
\[
\Psi_{\rm or} (u_1, u_2):= \mathbb{P} (\tau_{\rm or}<\infty),\quad \Psi_{\rm sim} (u_1, u_2):= \mathbb{P} (\tau_{\rm sim}<\infty)
\]
is reduced to finding univariate boundary crossing probabilities.  When $u_1/\delta_1\ge u_2/\delta_2$, that latter problem further reduces to simply computing usual univariate ruin probabilities. However, in the alternative case the situation is more interesting. For that latter case, assuming that $S(t)$ is a compound Poisson process with positive jumps satisfying the Cram\'er moment condition, Theorem~5 of~\cite{pal1} gives the exact asymptotics of $  \Psi_{\rm sim} (as, s)$ as $s\to \infty$, of which the nature depends on the (fixed) value of~$a > 0$. Namely, there exist a function $\kappa(a)>0$ and values $0\le a_1 < a_2\le \infty$ such that
\begin{equation}
\label{Palm_thm}
\Psi_{\rm sim} (as, s) = (1+o(1)) \left\{
\begin{array}{ll}
C(a) e^{-\kappa (a) s}, & a\not \in (a_1, a_2),\\
C(a) s^{-1/2 }e^{-\kappa (a) s}, & a \in (a_1, a_2),
\end{array}
\quad s \to \infty.
\right.
\end{equation}
However,  the approach from~\cite{pal1} does  not work in the case of   non-degenerate structure of the claim process $(S_1(t), S_2(t))$, and so the general problem of finding the exact ruin probability asymptotics remained open.

We extend the asymptotics of the simultaneous ruin probability derived in Theorem~5 in~\cite{pal1}   to a much more general class of  $d$-variate, $d \ge 2,$ Sparre Andersen--type ruin models, in which there are $d$ companies receiving premiums at the respective constant rates $c_i, i = 1, \ldots,d$, and  the claim events occur at the ``event times" $\tau(j)$, $j\ge 1$, in  a renewal process $N(t)$ (with i.i.d.\ inter-claim times $\tau(j)-\tau (j-1)>0$, $\tau (0):=0$). For the $j$-th claim event the amount company $i$ has to pay is the $i$-th component of a $d$-variate random vector $\bv{J}(j):=(J_1(j), \ldots, J_d(j))$ with a general (light tail) distribution, the vectors $(\tau(j)-\tau(j-1), \bv{J}(j)), j \geq 1,$ forming an i.i.d.\ sequence. Recall that in Theorem~5 from~\cite{pal1}, one had $\bv{J}(j)=(\delta_1, \delta_2)J(j)$  for an i.i.d.\ sequence of claims $J(j),$ $N(t)$ being a Poisson process independent of the $J(j)$'s.

We achieve that by reducing  the problem to finding the asymptotic behavior of the hitting probability of a remote set by an embedded random walk~(RW). The latter problem was solved in~\cite{mainpaper}, but only in the case when the boundary of that set is smooth at the ``most probable (hitting) point" of that set by the RW. In that case, the asymptotics of the hitting probability are of the form represented by the first line on the right-hand side of~\eqref{Palm_thm}.

The main  contribution of the present work  is an extension of the multivariate large deviation techniques from~\cite{mainpaper} to the cases where the boundary of the remote set is not smooth at the ``global  most probable point" (GMPP; for the formal definition thereof, see the text after~\eqref{mpt} below), but, rather, that point is located at the ``apex of the corner" on the boundary. It is in such cases that the hitting probability asymptotics (in the bivariate case) are of the form shown in the second line on the right-hand side of~\eqref{Palm_thm}. The condition $a\in (a_1, a_2)$ in~\eqref{Palm_thm} is equivalent to the GMPP being at the ``corner" of the remote quadrant (of which the hitting will mean simultaneous ruin in the bivariate case), while $a\not\in [a_1, a_2]$ corresponds to the GMPP being on one of the sides of the quadrant, which is  the ``smooth boundary case" dealt with in~\cite{mainpaper} (the boundary cases $a=a_i,$ $i=1,2,$ correspond  to the situations discussed in Remark~\ref{rem3} below). In Remark~\ref{R1} we explain the ``genesis" of the power factor in the asymptotics in the case when the GMPP is at the ``corner point" of the target set. It turns out that in the case $d>2$ there is a whole spectrum of different power factors that can appear in front of the exponential factor for the asymptotic representation of the hitting probability, depending on the dimensionality of the ``target set" boundary component to which the GMPP belongs,  see Remark~\ref{re2.2} below.

To explain in more detail,  let
\[
Q^+ := \{
\bv{x}=(x_1, \ldots, x_d) \in \mathbb{R}^d: x_j > 0,\ 1 \le j \le d \}.
\]
Next note that, in the above-mentioned $d$-variate Sparre Andersen--type model,   the simultaneous ruin event is equivalent to the bivariate RW 
 \begin{equation}\label{sn}
\bv{S}(n) := \sum_{j = 1}^{n}\bxi(j), \quad n = 0,1,2,\ldots,
 \end{equation}
with i.i.d.\ jumps  $\bxi(j) := \bv{J}(j) - \bv{c}(\tau(j) - \tau(j-1)), $ $ j \geq 1,$    $ \bv{c}:= (c_1, \ldots, c_d), $ hitting the set  $\bv{u} + {\rm cl}(Q^+), \bv{u} := (u_1, \ldots, u_d) \in Q^+$ being the vector of initial reserves of the $d$ companies (here and in what follows, ${\rm cl}(V)$ and ${\rm int} (V)$ stand for the closure and interior of the set $V$, respectively):
\[
\{\tau_{\rm sim} <\infty\} =  \{\eta (  \bv{u} + {\rm cl}(Q^+)) <\infty \}, \quad\mbox{where}\quad \eta (V) :=\inf\{n\ge 0: \bv{S}(n)\in V\}
\]
is the first hitting time   of the  Borel set $V \subset {\mathbb R}^d $ by the RW~$\bv{S}$. Further assuming that $\bv{u}= s\bv{g}$ for a fixed $\bv{g}\in Q^+$ and an $s>0$,  we see that $\bv{u} + {\rm cl}(Q^+) = sG,$ where
\[
G:=\bv{g}+ {\rm cl}(Q^+).
\]
Therefore the problem of finding the asymptotics of $\psi_{\rm sim} (  s\bv{g})$ as  $s\to \infty$ dealt with in Theorem~5 of~\cite{pal1} is   reduced to a special case of the main problem considered in~\cite{mainpaper}, i.e., computing  the asymptotics of  the probability
\begin{align}
\label{PsG}
\mathbb{P}(\eta (s G) <\infty), \quad s\to\infty.
\end{align}

However, as we already said, the main condition imposed in~\cite{mainpaper} on the admissible sets $G$ in~\eqref{PsG} was that the boundary of  $sG$ at the  ``most probable point" of that set is smooth (for precise definitions, see pp.\,248, 256 in~\cite{mainpaper}). This is not satisfied   in the most interesting case of  our ruin problem, where that point is located at the ``corner" $s\bv{g}$ of the set $sG$ (which means, roughly speaking, that  given that the RW $\bv{S}$ eventually hits $sG$, it is most likely that it does that in vicinity of that point~$s\bv{g}$). Thus the results of~\cite{mainpaper} are not  applicable in that case. In this paper, we extend them to such situations, obtaining asymptotics for~\eqref{PsG} of the form somewhat different from those in the ``smooth boundary case''. In particular, in the case $d = 2$ our result implies the relation in the second line in~\eqref{Palm_thm} for our Sparre Anderson--type model.

Roughly speaking, the asymptotics of~\eqref{PsG} in the $d$-dimensional case, when the boundary of $G$ is smooth in vicinity of the most probable point,  was derived in~\cite{mainpaper} as follows.  Let
\begin{equation}
\label{deltay}
  \Delta[\bv{y}) := \prod_{j = 1}^d [y_j, y_{j} + \Delta)
\end{equation}
be the cube with the ``left-bottom'' corner $\bv{y}$ and edge length $\Delta > 0$.  Starting with the representation
\begin{equation}
\label{sum_sec1}
\mathbb{P}\big( \eta(sG) < \infty \big) = \sum_{n \geq 1} \mathbb{P}\big( \eta(sG) = n \big),
\end{equation}
one computes the value of the summands on the RHS of~\eqref{sum_sec1} by summing up the terms of the form
\begin{equation}
\label{strongmarkov}
\mathbb{P}\big(\eta(sG) = n \big| \bv{S}(n) \in \Delta[\bv{y}) \big)\mathbb{P}\big( \bv{S}(n) \in \Delta[\bv{y}) \big)
\end{equation}
over $\bv{y}$-values on a $\Delta$-grid in a half-space (used instead of~$sG$ when $s$ is large, which is possible since the boundary of the set is smooth in vicinity of the most probable point). The second factor in~\eqref{strongmarkov} is evaluated using the integro-local large deviation theorem, whereas the first one can be computed using the smoothness of the boundary $\partial G$ by reducing the problem to evaluating the distribution of the global minimum of a one-dimensional RW with a positive trend. Finally, the sum on the RHS of~\eqref{sum_sec1} is computed using the Laplace method~\cite{erd_asy}.

However, a direct implementation of the above scheme in our case encounters serious technical difficulties (in particular, there is no above-mentioned  reduction to the univariate problem when computing the first factor in~\eqref{strongmarkov}). That may explain why~\cite{mainpaper} only dealt with the smooth boundary case. In the present  paper, we   employ a more feasible approach which introduces an auxiliary half-space $\widehat{H} \supset G,$ $\partial \widehat{H} \cap G = \{ \bv{g}\},$ such that the logarithmic asymptotics of $\mathbb{P}\big( \eta(s\widehat{H}) < \infty \big)$ as $s\to\infty$ are the same as for $\mathbb{P}\big(\eta(sG) < \infty \big)$   and the ``most probable points" for $s\widehat{H}$ and $sG$ both coincide with $s\bv{g}$ (cf.\ Lemmata~\ref{halfspacempp1} and~\ref{DG} in Section~\ref{proofs_section}). Then, in Theorem~\ref{mainthm1} below,  we use the approach from~\cite{mainpaper} together with the integro-local large deviation theorem and the total probability formula to derive the fine asymptotics for the probabilities of the form
\begin{equation}
\label{mid_step}
\mathbb{P}\big( \eta(sG) <\infty, \eta(s\widehat{H} ) = n, \bv{S}(n) \in s\bv{g} + \Delta[\bv{y}) \big).
\end{equation}
Next we partition $\widehat{H}$ into a narrow half-cylinder with the generatrix orthogonal to $\partial \widehat{H},$ that covers the ``very corner of $sG$'', and its complement in $\widehat{H}$ (as shown in Fig.~\ref{figure:3}). The main contribution to $\mathbb{P}\big( \eta(sG) < \infty \big)$ is computed by ``integrating''~\eqref{mid_step} in $\bv{y}$ over that half-cylinder and then by summing up the resulting expressions (denoted by $P_{3,n}$ in the proof of Theorem~\ref{finalthm} in Section~\ref{proofs_section}, see~\eqref{53a}, \eqref{p3n_sum}) over $n$ using the Laplace method. The total contribution of the terms~\eqref{mid_step} with $\bv{y}$ outside that half-cylinder (which is equal to the sum $P_1+P_2,$ cf.~\eqref{PPP}) is shown to be negligibly small compared to the above-mentioned main term.

To give  precise definitions of the key concepts like the ``most probable point" and exact formulations of our results, we will need to introduce some notations and a number of important concepts from the large deviation theory for RWs with i.i.d.\ jumps in $\mathbb{R}^d$. This is done in Section~\ref{Sec_2}. That section also contains a summary of the key properties of the deviation rate functions defined and discussed there, some auxiliary constructions and the main result (Theorem~\ref{finalthm}) of the paper as well.  Further auxiliary constructions and assertions are presented in Section~\ref{proofs_section}, together with the proof of Theorem~\ref{finalthm}.

\section{Some Preliminaries and the Main Result}
\label{Sec_2}

In this section, we will present and discuss the key concepts needed for the Cram\'{e}r large deviation theory, in particular, the first and second (deviation) rate functions. For introduction to large deviation theory for univariate RWs and main properties of the first rate function, see Chapter~9 in~\cite{mainbook2}.

Unless stated otherwise, all the concepts and properties discussed in this section were introduced and/or established in~\cite{paper20, paper23}. Moreover, we will introduce three important conditions assumed to be met in the main theorem that we state at the end of the section. We conclude this section with remarks commenting on the difference between the forms of the asymptotics of the hitting probabilities in the smooth and non-smooth cases, and also on possible extensions of our main result.



For vectors $\bv{u} = (u_1, \ldots, u_d)$, $\bv{v} = (v_1,  \ldots, v_d) \in \mathbb{R}^d$, $d \geq 2$, we set $\langle \bv{u}, \bv{v} \rangle := \sum_{i=1}^d u_iv_i$ and $ \|\bv{v}\| := \langle \bv{v}, \bv{v} \rangle ^{1/2}$. For a function $f \in C^1(S),$ $S$ being an open subset of~$\mathbb{R}^d$,  and~$\bv{x}\in \label{key}S,$ we denote by   $  f'(\bv{x}) := \nabla_{\sbv{x}}f(\bv{x}),$ where $ \nabla_{\sbv{x}} := \big(\frac{\partial }{\partial x_1}, \ldots, \frac{\partial }{\partial x_d}\big),$ the gradient of $f$  at  $\bv{x}$.  (On a couple of occasions, where it may be unclear with respect to what variable the gradient is computed, we will still have to use the nabla with the respective subscript.)
By $f''$ we denote the Hessian of the function $f \in C^2(S):$ using $^T$ for transposition,
\begin{equation*}\label{2nddev}
f''(\bv{x}) := \nabla_{\sbv{x}}^T\nabla_{\sbv{x}}\, f(\bv{x}).
\end{equation*}

Let $\bxi$ be a  random vector in $\mathbb{R}^d$ satisfying the following condition:
\smallskip

 [$\textbf{C}_1$]~{\em The distribution $F$ of $\bxi$ is non-lattice and there is no hyperplane $K = \{ \bx :  \lag \bv{a}, \bx \rag = c \} \subset \mathbb{R}^d$ such that $F(K) = 1.$}
\smallskip

The moment generating function  of $\bv{\xi} \in \mathbb{R}^d$ is denoted by
\begin{equation*}\label{cm}
 \psi(\bv{\lambda}):= \mathbb{E}e^{\langle \sbv{\lambda}, \sbv{\xi} \rangle} = \int e^{\langle \sbv{\lambda}, \sbv{x} \rangle} F(d\bv{x}), \hspace{ 6 mm} \bla \in \mathbb{R}^d.
\end{equation*}
Let $\Theta_{\psi} := \{ \bv{\lambda} \in \mathbb{R}^d:  \psi(\bv{\lambda}) < \infty \}$ be the    set on which~$\psi$ is finite. It is well known that  $\Theta_{\psi}  $ is convex. We will need the following Cram\'{e}r moment condition imposed on $F$:
\smallskip

[$\textbf{C}_2$]~\textit{$ \Theta_{\psi}$ contains a non-empty open set}.
\smallskip

Under condition [$\textbf{C}_2$], for a fixed $\bla \in \Theta_{\psi}$, the Cram\'{e}r transform $F_{\sbv{\lambda}}$ of the distribution $F$ for that $\bv{\lambda}$ is defined as the probability distribution given by
\begin{equation*}
\label{cramtransF}
 F_{\sbv{\lambda}} (W) := \frac{\mathbb{E} (e^{\langle \sbv{\lambda},\sbv{\xi}  \rangle} ; \bv{\xi} \in W)}{\psi(\bv{\lambda})}, \hspace*{4 mm} W \in \mathcal{B}(\mathbb{R}^d),
\end{equation*}
where $\mathcal{B}(\mathbb{R}^d)$ is the $\sigma$-algebra of Borel subsets of $\mathbb{R}^d$  (see e.g.~\cite{paper7,paper8}). Denote by $\bxi_{(\sbla)}$ a random vector with distribution $F_{\sbla}$.

The  first rate function $\Lambda(\bv{\alpha})$ for the random vector $\bv{\xi}$ is defined as
\begin{equation}\label{rf1}
   \Lambda(\bv{\alpha}) :=   \sup_{\sbv{\lambda} \in \Theta_{\psi}} (\langle \bv{\alpha},\bv{\lambda} \rangle - \ln\psi(\bv{\lambda})), \hspace*{4 mm} \bv{\alpha} \in \mathbb{R}^d,
\end{equation}
which is the Legendre transform of the cumulant function $ \ln\psi(\bv{\lambda})$.
For $\bal \in \mathbb{R}^d$, denote by $\bla(\bal)$ the vector $\bla$ at which the upper bound in \eqref{rf1} is attained (when such a vector exists, in which case it is always unique):
\begin{equation*}
 \Lambda(\bv{\alpha})  =\langle \bv{\alpha},\bv{\lambda}(\bal)\rangle - \ln\psi(\bv{\lambda}(\bal)).
\end{equation*}

Define the Cram\'{e}r range $\Omega_{\Lambda}$ for $F$ as the set of all vectors that can be obtained as the expectations of the Cram\'{e}r transforms $F_{\sbla}$ of $\bxi$ for $\bla \in {\rm int}(\Theta_{\psi})$:
\begin{equation*}
\Omega_{\Lambda} := \bigg\{ \bal = \frac{\psi'(\bla)}{\psi(\bla)} \equiv (\ln\psi(\bla))', \, \bla \in {\rm int}(\Theta_{\psi})  \bigg\}.
\end{equation*}
The rate function $\Lambda$ is convex on $\mathbb{R}^d$ and strictly convex and analytic on $\Omega_{\Lambda}$.  Moreover, for $\bal\in\Omega_{\Lambda}$, one has (cf.~\cite{paper20})
\begin{equation}
\label{bla(bal)}
\bla(\bal) = \Lambda'(\bal).
\end{equation}

Introduce notations $F^{(\sbal)} := F_{\sbla(\sbal)} $ and $\bxi^{(\sbal)} := \bxi_{(\sbla(\sbal))} $ and define
\begin{equation}
\label{12a}
 \bv{S}^{(\sbal)}(n) := \sum_{i=1}^n \bv{\xi}^{(\sbal)}(i), \qquad n =0, 1, \ldots ,
\end{equation} where $\bxi^{(\sbal)}(i)$ are independent copies of $\bxi^{(\sbal)}$.  For $\bal \in \Omega_{\Lambda}$, one can easily verify that
\begin{equation*}
\mathbb{E}\, \bv{\xi}^{(\sbv{\alpha})}
 = \ln \psi(\bla)' \big|_{\sbla = \sbla(\sbal)} = \bv{\alpha},  \quad
 \mbox{ Cov }\bv{\xi}^{(\sbal)} = \ln \psi(\bla)'' \big|_{\sbla = \sbla(\sbal)} = \big( \Lambda''(\bv{\alpha}) \big)^{-1}.
\end{equation*}
Denote by
\begin{equation*}\label{varalpha}
\sigma^2(\bv{\alpha}) := \mbox{det}\,  \mbox{Cov}\,\bv{\xi}^{(\sbal)} = \text{det}\, \big( \Lambda''(\bal) \big)^{-1}
\end{equation*} the determinant of the covariance matrix of $\bv{\xi}^{(\sbal)}$.

The probabilistic interpretation of the first rate function is as follows (see e.g.\ \cite{mainpaper}): for any $\bal \in \mathbb{R}^d$, letting $U_{\varepsilon}(\bal)$ denote the $\varepsilon$-neighborhood of $\bal$,
\begin{equation*}
\Lambda(\bal) = -\lim_{\varepsilon \rightarrow 0}\lim_{n \rightarrow \infty} \frac{1}{n}\ln \mathbb{P}\Big(\frac{\bv{S}(n)}{n} \in U_{\varepsilon}(\bal) \Big),
\end{equation*}
Accordingly, for a set $B \subset \mathbb{R}^d$,  any point $\bal \in B$ such that
\begin{equation}\label{MPPdef}
\Lambda(\bal) = \inf_{\sbv{v} \in B}\Lambda(\bv{v})
\end{equation}
is called the most probable point (MPP) of~$B$.  If such an   $\bal$ is unique, we denote it by
\begin{equation}
\label{MPP}
\bal[B] := \argminA_{\sbv{v} \in B}\Lambda(\bv{v}).
\end{equation}

Since $\Lambda$ is convex,  $\Lambda(\bal) \geq 0$ for any $\bal \in \mathbb{R}^d$ and $\Lambda(\bal) = 0$ if and only if $\bal = \mathbb{E} \bxi,$ for $\mathbb{E} \bxi \notin B$ one always has
\begin{equation}
\label{aldB}
\bal[B] = \bal[\partial B],
\end{equation}
so that the MPP in that case is on the boundary of $B$.

The concept of the MPP for the  set $G$ is  related to the behavior of   $\mathbb{P}\big(\SU{n} \in sG \big)$ as $s \rightarrow \infty$, $n \asymp s$. However, we are interested in the probability of the event that the trajectory $\{\SU{n} \}_{n \geq 1}$ ever hits $sG$. To deal with that problem, we need to introduce the concept of  the second rate function $D$ defined in \cite{paper23} as follows:
\begin{equation}
\label{Dv}
D(\bv{v}) := \inf_{t > 0} \frac{\Lambda(t\bv{v})}{t}, \quad \bv{v} \in \mathbb{R}^d.
\end{equation}
This function  admits an alternative representation   (see Theorem~1 in~\cite{paper23}) of the form
\begin{equation}\label{Dv1}
D(\bv{v})= \sup \{ \lag \bv{\lambda},\bv{v} \rag : \psi (\bv{\lambda})\le 1 \}, \quad \bv{v} \in \Omega_\Lambda. 
\end{equation}
The following    key properties  [$\textbf{D}_1$]--[$\textbf{D}_4$] of the second rate function will be used below. The first one is an immediate consequence of representation~\eqref{Dv1}:

\medskip
 [$\textbf{D}_1$]~~The function $D$ is convex on $\mathbb{R}^d$.

\medskip Now introduce $t(\bv{v})$ as  the point at which the infimum  in \eqref{Dv} is attained:
\begin{equation*}\label{uniquer}
D(\bv{v}) = \frac{\Lambda( t(\bv{v})\bv{v})}{t(\bv{v})}.
\end{equation*}
The next property is established in Theorem~2 in~\cite{paper23}.

\medskip

 [$\textbf{D}_2$]~~For any $\bv{v} \in \Omega_{\Lambda},$ the point $t(\bv{v})\bv{v}$ is an analyticity point of $\Lambda$ and $t(\bv{v})$ is unique.

\medskip

It will also be convenient to consider the reciprocal quantity
\begin{equation*}\label{u=1/t}
u := \frac{1}{t}.
\end{equation*}
For $\bv{v} \in \mathbb{R}^d$ and $B \subset \mathbb{R}^d$, put
\begin{equation}\label{doubleinf1}
D_u(\bv{v}) := u \Lambda \Big( \frac{\bv{v}}{u}\Big), \qquad D_u(B) := \inf_{\sbv{v} \in B} D_u(\bv{v}),
\end{equation} and let
\begin{equation}
\label{doubleinf}
 D(B) := \inf_{u > 0}D_u(B) = \inf_{u > 0}\inf_{\sbv{v} \in B}  u \Lambda \Big( \frac{\bv{v}}{u}\Big).
\end{equation}

The value
\begin{equation}
\label{mpt}
u_B := \argminA_{u > 0} D_u(B)
\end{equation}
is called the most probable time (MPT) for the set $B \subset \mathbb{R}^d$.
We put
\begin{equation}
\label{rB}
r_B := 1/u_B.
\end{equation}
The reason for calling $u_B$ the MPT is as follows. The problem of hitting the remote set $sB$ ($s \rightarrow \infty$) by the RW $\{\SU{n} \}_{n \geq 1}$ can be re-stated in the scaled time-space framework  as that of hitting the original set $B$ by the process $\{s^{-1}\bv{S}(\floor{su})\}_{ u > 0}$. Then, given that that continuous time process hits $B$, it is most likely to do so at a time close to~$u_B$.

We refer to the point $\bv{b}\in B$ such that $D(\bv{b})=D (B)$ as the global  MPP (GMPP) for the set~$B$. The probabilistic meaning of the GMPP  is that, in a setting where $s\to\infty,$  if our RW ever hits the set~$sB$, it is most likely that it will do that in the vicinity of~$s\bv{b}$ (i.e., within a  distance~$o(s)$ therefrom).

\smallskip

 We will need two more properties of the function $D$.

 \medskip

[$\textbf{D}_3$]~~If one has $D(B) = D(\bv{b})$ for a $\bv{b} \in B \subset \mathbb{R}^d$, then
\begin{equation}\label{D7}
 D' (\bv{v})\big|_{\sbv{v} = \sbv{b}} =  \Lambda' (\bal)\big|_{\sbal = r_B\sbv{b}} = \bla(r_B\bv{b}).
\end{equation}

The latter equality in \eqref{D7} is the  known key property~\eqref{bla(bal)} of the rate function $\Lambda$. To  prove the former one, note that,  from [$\textbf{D}_2$] and the implicit function theorem, one has
\begin{equation*}
 \frac{\partial}{\partial u}D_u(\bv{v})\Big|_{u = u(\sbv{v})} = 0, \quad u(\bv{v}) := \frac{1}{t(\bv{v})}.
\end{equation*}
Therefore,  as $ D (\bv{v})=D_{u(\sbv{v})}(\bv{v}), $ using the chain rule results in
\begin{align*}
 D'  (\bv{v}) &= \bigl( D_{u(\sbv{v})}(\bv{v})\bigr)'
=
 \frac{\partial}{\partial u}D_u(\bv{v})\Big|_{u = u(\sbv{v})} u'(\bv{v})
  + \bigl( \nabla_{\sbv{v}}D_u(\bv{v})\bigr)\Big|_{u = u(\sbv{v})} \\
&= \Big(\nabla_{\sbv{v}} u \Lambda \Big(\frac{\bv{v}}{u}\Big)\Big) \Big|_{u = u(\sbv{v})}
=
   \Lambda'(\bal)\Big|_{\sbal = \sbv{v}/u(\sbv{v})}.
\end{align*}
As $u(\bv{b}) = u_B = 1/r_B$ by assumption,  property [$\textbf{D}_3$] is proved.

\medskip

[$\textbf{D}_4$]~~$D_u(\bv{v}) = u \Lambda (\bv{v}/u )$ is a convex function of $(u, \bv{v}) \in \mathbb{R}^+\times \mathbb{R}^d.$

\medskip

To prove this property, let $(u_1, \bv{v}_1), (u_2, \bv{v}_2) \in \mathbb{R}^+ \times \mathbb{R}^d$ be any two points in the time-space, $a \in (0,1)$.
The function $\Lambda$ is convex, so that
 \begin{equation*}
 \Lambda(p\bal_1 + (1-p)\bal_2) \le p\Lambda(\bal_1) + (1-p)\Lambda(\bal_2), \quad p \in (0,1), \quad \bal_1, \bal_2 \in \Omega_{\Lambda}.
 \end{equation*} By choosing $p := \frac{au_1}{au_1 + (1-a)u_2}$, $\bal_1 := \bv{v}_1/u_1$ and $\bal_2 := \bv{v}_2/u_2$ in the above inequality and multiplying both sides by $au_1 + (1-a)u_2,$ we get
\begin{equation}
\label{convexineq}
(au_1+(1-a)u_2) \Lambda \Big(\frac{a_1\bv{v}_1 + (1-a)\bv{v}_2}{au_1 + (1-a)u_2}\Big) \leq au_1\Lambda \Big( \frac{\bv{v}_1}{u_1} \Big) + (1-a)u_2\Lambda \Big( \frac{\bv{v}_2}{u_2} \Big),
\end{equation} which establishes the desired convexity. Property [$\textbf{D}_4$] is proved.

\medskip

For $r > 0$, let
\begin{equation}\label{type-one-ls}
L(r) := \{ \bv{v} \in \mathbb{R}^d: \Lambda(\bv{v}) = \Lambda(\bal(r)) \},
\end{equation}
be the level surface (line when $d = 2$) of  $\Lambda$ that passes  through the point
\begin{equation}
\label{balrg}
\bal(r) := \bal[rG]
\end{equation}
(see \eqref{MPP}; we assume here that there exists a unique point $\bal$ satisfying~\eqref{MPPdef} with $B=rG$) and introduce the respective superlevel set
\begin{equation*}
\widehat{L} (r) := \{ \bv{v} \in \mathbb{R}^d: \Lambda(\bv{v}) \geq \Lambda(\bal(r)) \}.
\end{equation*}

\begin{lemma}\label{uniquemmp}
Let $r > 0$. If there is an $\bal_0 \in \Omega_{\Lambda} $ such that $\bal_0$ is an  MPP for the set $rG,$  then this MPP is unique for $rG$:
\begin{equation}\label{unique}
\{ \bal_0\} = \{\bal(r)\} = L(r) \cap rG.
\end{equation}
\end{lemma}

The proof of Lemma~\ref{uniquemmp} is given in Section 3.

Consider the following condition that depends on parameter $r > 0$:

\medskip

[\textbf{C}$_3(r)$]~\textit{One has }
\begin{equation*}
%
\bla(r\bv{g}) \in Q^+, \quad r\bv{g} \in \Omega_{\Lambda}, \quad \lag \mathbb{E} \bxi, \bla(r\bv{g}) \rag < 0.
\end{equation*}
The first part of the condition means that the ``external'' normal vector to the   level surface  of the convex function  $\Lambda$    at the point $r\bv{g}$  points inwards~$rG$, which means that the vertex $r\bv{g}$ is an MPP for $rG$. Under the second part of the condition, this MPP  for  $rG$ is unique by Lemma~\ref{uniquemmp}: $r\bv{g}= \bal(r),$ so that $\bla(r\bv{g})$ coincides with the vector
\begin{equation}
\label{normaln}
 \bv{N}(r) :=    \Lambda' (\bal)\big|_{\sbal = \sbal(r)}= \bla(\bal(r)), \quad r > 0,
\end{equation}
which is a normal vector to the  level surface $L(r)$ at the point $\bal(r)$ pointing inwards~$\widehat{L}(r)$ (the above definition of $\bv{N}(r)$   makes sense whenever $rG$ has a unique MPP).  Since always $ \bv{N}(r) \in {\rm cl}(Q^+)$, the first  part of [\textbf{C}$_3(r)$]   excludes the case when the normal  to $L(r)$ at the point $ \bal(r)=r\bv{g}$ belongs to the boundary of the set $rG$.

The main result of the present paper is the following assertion.

\begin{theorem}
\label{finalthm}
If conditions  {\rm [\textbf{C}$_1$], [\textbf{C}$_2$]}    and {\rm [\textbf{C}$_3(r_G)$]} are satisfied, where $r_G$ is defined by~\eqref{rB} with $B=G$, then
\begin{equation}
\label{lastassertion}
\mathbb{P}\big( \eta(sG) < \infty \big) = As^{-(d-1)/2}e^{-sD(G)}(1+o(1))
\quad as\quad s \rightarrow \infty,
\end{equation}
where the value of the constant $A \in (0, \infty)$ is given in~\eqref{AAA}.

\end{theorem}

\begin{remark}\label{R1}
{\rm
In the ``smooth case'', when the boundary of $ G$ is twice continuously differentiable in the vicinity of the GMPP  (the latter was defined  after~\eqref{rB}),  the exact asymptotics for the hitting probability was shown to have the form
\begin{equation}
\label{smoothlastassertion}
\mathbb{P}\big( \eta(sG) < \infty  \big) = B e^{-sD(G)}(1 + o(1)), \quad s \rightarrow \infty,
\end{equation}
where the constant $B > 0$ (depending on $F$ and $G$) can be written down explicitly (see Theorem~7 in \cite{mainpaper}). Thus, the qualitative difference between the asymptotics~\eqref{lastassertion} in the case of the orthant $G$ with the GMPP at its vertex and the asymptotics \eqref{smoothlastassertion} in the ``smooth case'' is the presence of the power factor $s^{-(d-1)/2}$ in the former formulation (cf.\ the factor $s^{-1/2}$ in the second line of \eqref{Palm_thm}, the asymptotics of the ruin probability in the special bivariate case from \cite{pal1}).

The presence of that power factor can be roughly explained as follows. The distribution of the location of the first hitting point of the auxiliary half-space $s\widehat{H}(r_G) \supset sG$ (defined below, see~\eqref{Hr} and~\eqref{31a}) is close to the normal law on its boundary $sH(r_G)$ with the mean point at $s\bv{g}$ and covariance matrix proportional to~$s^{1/2}$ (see Corollary $3.2$ in \cite{mainpaper}). However, the RW $\bv{S}$ will only have a noticeable chance of hitting $sG$ at or after the time when it hits $s\widehat{H}(r_G)$ if the ``entry point'' to $s\widehat{H}(r_G)$ is basically in a finite neighborhood of the vertex point $s\bv{g}$. It is the integration over that neighborhood with respect to the above-mentioned ``almost normal" distribution ``of the scale $s^{1/2}$'' that results in the additional factor $s^{-(d-1)/2}$ on the RHS of~\eqref{lastassertion}.
}
\end{remark}

\begin{remark}
\label{re2.2}
{\rm
One can consider, in a similar way, the case where the GMPP  neither lies on the face of the orthant (which would be the ``smooth case" dealt with in  \cite{mainpaper}) nor is the vertex thereof (our case), but lies on an $m$-dimensional ($1\le m<d-1$) component of the orthant boundary. It is not hard to see from our proofs that the hitting probability asymptotics in such a case will be ``intermediate" between~\eqref{smoothlastassertion} and~\eqref{lastassertion}, with the power factor $s^{-(d-m-1)/2}$.

A rough explanation of that is similar to the one given in Remark~\ref{R1}. In that case, the distribution of the location of the first hitting point of the auxiliary half-space (of which the boundary will now contain the respective $m$-dimensional component of the orthant boundary) will again be close to the normal law on the boundary of that half-space, with the covariance matrix proportional to~$s^{1/2}$. But now, to have a noticeable chance of hitting the set $sG$, the ``entry point'' to $s\widehat{H}(r_G)$ should be within a ``short distance" from that $m$-dimensional component of the orthant boundary (rather than from the GMPP itself). So now we will have to integrate with respect to  the above-mentioned ``almost normal" distribution over a subset of the  hyperplane which is ``bounded in  $(d-m-1)$ directions", hence the resulting power factor.
}
\end{remark}

\begin{remark}
\label{rem3}
{\rm
If conditions  [\textbf{C}$_1$], [\textbf{C}$_2$], [\textbf{C}$_3(r_G)$] are met except for the last assumption that $\lag \mathbb{E} \bxi, \bv{N}(r_G) \rag < 0,$ we still have a large deviation situation provided that $\mathbb{E} \bxi \notin~{\rm cl}(Q^+).$ In that case, $\bv{g}$ will still be the GMPP for $G$, but the asymptotics of~\eqref{PsG} will be of the same form~\eqref{lastassertion} as in the smooth boundary case (except for the value of the constant $B$). The reason for that will be clear from the proof of Theorem~\ref{finalthm} (more precisely, from its part  dealing with bounding the term $P_1$). Roughly speaking, what happens in that case is that if the RW $\bv{S}$ enters the auxiliary half-space $s\widehat{H}(r_G)$ in the sector from which one can ``see'' the set $sG$ along the rays with the directional vector $\mathbb{E} \bxi$, then the RW will eventually hit $sG$ with probability bounded away from zero. The probability of hitting that part of $s\widehat{H}(r_G)$ differs from the probability of hitting the ``smooth case'' set $s\widehat{H}(r_G)$ by basically a constant factor.
}
\end{remark}

\begin{remark}
\label{rem4}
{\rm
As discussed at the beginning of this section, the RHS of~\eqref{lastassertion} gives the asymptotics of the simultaneous ruin probability $\Psi_{{\rm sim}}(s\bv{g})$ as $s \to \infty$ in the $d$-dimensional extension of the problem from~\cite{pal1}, under conditions [\textbf{C}$_1$], [\textbf{C}$_2$], [\textbf{C}$_3(r_G)$].  In the case of an alternative location of the GMPP, the asymptotics of  $\Psi_{{\rm sim}}(s\bv{g})$ can be obtained from the main result of \cite{mainpaper} (when GMPP is on the face of~$G$) or arguing as indicated in Remark~\ref{re2.2} (in all other cases).
}
\end{remark}

\begin{remark}
{\rm
Our result could also be extended to the case of a more general set~$G$, with the property that the GMPP for hitting that set by our RW is at a ``vertex" on~$\partial G$.  Here is a possible set (i)--(iv) of conditions for such an extension.

(i)~$G\cap  \{ t \mathbb{E} \bxi: t\ge 0\}=\varnothing.$ (ii)~There is a  $\bv{g}\in\partial G$ which is the unique GMPP of~$G$. (iii)~There exist $\varepsilon,\delta >0$ such that
$
D(\bv{g})\equiv  D (G) < D(G\setminus U_{\varepsilon} (\bv{g}))   - \delta .
$
(iv)~Denote by
\[
C_\theta (\bv{b}):=\Bigl\{\bv{v}\in \mathbb{R}^d:\arccos
\frac{\langle\bv{v}, \bv{b}\rangle}{\|\bv{v}\|\| \bv{b} \|} \le \theta\Bigr\}, \quad \theta \in (0,\pi/2),
\]
a circular cone  in $\mathbb{R}^d$ with the axis direction vector $\bv{b} ,$ opening angle  $2\theta$  and apex at~$\bv{0}$, and by $\bv{\zeta} $ the unit normal vector to the   level surface of $\Lambda$  passing through  the GMPP (see~\eqref{unitnormal}). Then there exist a $\bv{b} \in \mathbb{R}^d$ and values  $0<\theta_1 <\theta_2 <\pi/2$ such that
\[
 C_{\theta_1} (\bv{b})   \cap U_{\varepsilon} (\bv{0})
\subset
(G-g)\cap U_{\varepsilon} (\bv{0})
\subset
C_{\theta_2} (\bv{\zeta})\cap U_{\varepsilon} (\bv{0}) .
\]
Condition (iv) ensures that $\partial G$ is non-smooth at the GMPP $\bv{g}$, where it has a ``vertex'' with a positive solid angle at it. It is not very  hard to verify, using basically the same argument as the one in the proof of our Theorem~\ref{finalthm} (but with a number of appropriate changes) that $\mathbb{P} ( \eta(sG) < \infty )$ for such a   $G$ will also have asymptotics of the form~\eqref{lastassertion}.}
\end{remark}

\begin{remark}
{\rm Moreover, one can further extend the setup of our large deviation problem considering, instead of just ``inflated sets" $sG,$ other versions of ``remote sets". Such possible versions include  shifts $s\bv{g}+V$ for some fixed set $V\subset \mathbb{R}^d$ (which coincides with the inflated set $s(\bv{g}+V)$ in our special case when $V={\rm cl}\, (Q^+),$ but would be different from that set when $V$ is not a cone), combinations of shift and inflation which may, say, be of the form $ s\bv{g}+s^v V$ for some $v>0,$ and so on.   It appears that our approach would  also work for some of those other   settings, but the answers may be different in their form from both~\eqref{lastassertion} and~\eqref{smoothlastassertion}.}
\end{remark}

\section{Proofs}
\label{proofs_section}

For the reader's convenience, we start this section with a short list of notations that either have already been introduced or will   appear below and that are often used in the proofs. Next to the notations are the numbers of the displayed formulae with the resp.\ definitions:
\begin{itemize}
\item auxiliary hyperplanes and half-spaces:
$H(r)  = \bv{g} + H_0(r), $ $\widehat{H}(r) = \bv{g} + \widehat{H}_0(r)$
\eqref{Hr};

\item most probable points:
$\bal[B] = \argminA_{\sbv{v} \in B}\Lambda(\bv{v})$
 \eqref{MPP}; $\bal(r)  = \bal[rG]$ \eqref{balrg};  $\bbe(r) =\bal[r\widehat{H}(r_G)]$
\eqref{bber}; $\bv{\chi}  = n\bbe(s/n) - s\bv{g}$
\eqref{chi};

\item the second rate function and related objects:
 $D_u(\bv{v})  = u \Lambda  (  {\bv{v}}/{u} ),$  $D_u(B) = \inf_{\sbv{v} \in B} D_u(\bv{v})$ \eqref{doubleinf1}; $D(\bv{v})  = \inf_{u > 0}D_u(\bv{v})$ \eqref{Dv}; $
 D(B)  = \inf_{u > 0}D_u(B)$ \eqref{doubleinf};

\item most probable times: $u_B  = \argminA_{u > 0} D_u(B)$
\eqref{mpt}; $r_B  = 1/u_B$ \eqref{rB};

\item normals to the level surfaces of $\Lambda$: $ \bv{N}(r) =  \bla(\bal(r))$
\eqref{normaln}; $\bv{\zeta} = \frac{\boldsymbol{N }(r_G)}{\|\boldsymbol{N}(r_G) \|}$
\eqref{unitnormal};

\item the first hitting time  of the auxiliary half-space:
$\eta_s = \eta(s\widehat{H}(r_G) )$ (cf.~\eqref{PPP}; this formula also introduces probabilities
$P_j,$ $ j=1,2,3$).  
\end{itemize}
Finally, by  $c$ (with or without subscripts) we denote in this section  positive constants (possibly different within one and the same argument and depending on $F$ and $\bv{g}$).

In all the assertions below except Lemma~\ref{halfspacempp1} we always assume that conditions  [\textbf{C}$_1$], [\textbf{C}$_2$]    and [\textbf{C}$_3(r_G)$], where $r_G$ is defined by~\eqref{rB} with $B=G$,  are met.

The scheme of the proof of our main result was outlined in the second last paragraph of the Introduction.   At the first step, we will prove Lemma~\ref{uniquemmp}.

\begin{proof}[Proof of Lemma~\ref{uniquemmp}]
Suppose there is another MPP $\bal_1 \neq \bal_0$ for the set $rG$. Denote by $l$ the straight line segment with the end points $\bal_0$ and $\bal_1$. Since both $rG$ and the sublevel set $\widetilde{L}:= \{\bv{v} \in \mathbb{R}^d: \Lambda(\bv{v}) \leq \Lambda(\bal_0) \}$ are convex and ${\rm int}(\widetilde{L}) \cap (rG) = \varnothing$ (as $\Lambda(\bv{v}) < \Lambda(\bal_0)$ for any $\bv{v} \in {\rm int}(\widetilde{L})$), one must have $l \subset (rG) \cap \widetilde{L} = (rG) \cap \partial \widetilde{L}$. The latter relation implies that
\begin{equation}\label{star}
\Lambda(\bal) = \Lambda(\bal_0), \quad \bal \in l.
\end{equation} As $\bal_0$ belongs to the open set $\Omega_{\Lambda}$, there exists an  $\varepsilon \in (0, \|\bal_0 - \bal_1 \|)$ such that $U_{\epsilon}(\bal_0) \subset \Omega_{\Lambda}$, so that $\Lambda$ is strictly convex on $U_{\epsilon}(\bal_0).$ In particular, it is strictly convex on the segment $l \cap U_{\epsilon}(\bal_0)$, which contradicts to~\eqref{star}. Lemma \ref{uniquemmp} is proved.
\end{proof}

Next we  will construct  auxiliary half-spaces. Recall~\eqref{normaln} and let
\begin{equation*}\label{H0}
 H_0(r) := \{\bv{v} \in \mathbb{R}^d: \lag \bv{v}, \bv{N}(r) \rag = 0 \}, \quad  \widehat{H}_0(r) := \{\bv{v} \in \mathbb{R}^d: \lag \bv{v}, \bv{N}(r) \rag \geq 0 \}
\end{equation*} be the linear subspace orthogonal to $\bv{N}(r)$ and the ``upper'' half-space bounded by $H_0(r)$, respectively.
Their respective translations by the vector $\bv{g}$ we will denote by
\begin{equation}
\label{Hr}
H(r) := \bv{g} + H_0(r) \quad {\rm and } \quad \widehat{H}(r) := \bv{g} + \widehat{H}_0(r).
\end{equation}  Under condition [\textbf{C}$_3(r)$], one has
\begin{equation}
\label{31r}
rH(r) = r\bv{g} + rH_0(r) =  \bal(r) + H_0(r) \quad {\rm and} \quad r\widehat{H}(r) = \bal(r) + \widehat{H}_0(r).
\end{equation}

\begin{figure}[ht]	
	\centering
	\includegraphics[scale=0.45]{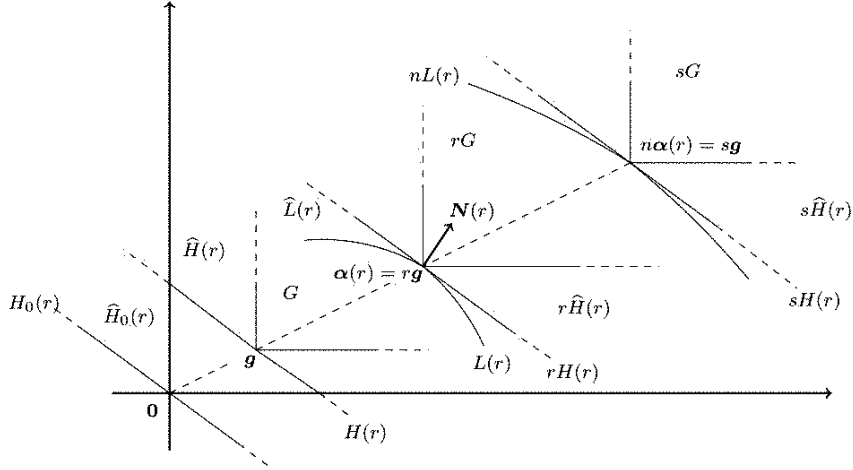}
	\vspace{-5mm}
	\caption{Auxiliary constructions: the level line $L(r)$ of $\Lambda,$ its scaled version $nL(r),$ the resp.\ tangent straight lines $rH(r)$ and $sH(r)$, and other related objects (case $d = 2, r = s/n$).}
	\label{figure:1}
\end{figure}

Since $\bal(r) \in L(r)$ by \eqref{type-one-ls} and $\bal(r) = r\bv{g}$ from condition  [\textbf{C}$_3(r)$],    one has $n\bal(r) = nr\bv{g} = s\bv{g} \in n L(r)$ (see Fig.\,1), when we choose
\begin{equation}
\label{31x}
r := \frac{s}{n},
\end{equation} where $s > 0$ is the parameter used to scale the set $G$ and $n \in \mathbb{N}$ will have the interpretation of the number of steps in the RW $\bv{S}$ (see \eqref{sn}).

Hence the sets
\begin{equation}
\label{31a}
 sH(r) = s\bv{g} + H_0(r) \quad {\rm and} \quad s\widehat{H}(r) = s\bv{g} + \widehat{H}_0(r),
\end{equation}
are, respectively, the tangent hyperplane to the scaled surface $nL(r)$ at the point $s\bv{g}$ and   the ``upper'' half-space bounded by $sH(r)$.

The role of the half-space $r\widehat{H}(r)$ is clarified in the next lemma,  which shows that the MPP for   $r\widehat{H}(r)$ coincides with the MPP for the scaled version $rG$ of the set~$G$.

\begin{lemma}\label{halfspacempp1}
If conditions  {\rm [\textbf{C}$_1$]}, {\rm [\textbf{C}$_2$]}  and {\rm [\textbf{C}$_3(r)$]} are   satisfied for an $r > 0$, then
\begin{equation*}\label{va}
\bal\big[r\widehat{H}(r)\big] = \bal(r) = r\bv{g}.
\end{equation*}
\end{lemma}

\begin{proof}[Proof of Lemma \ref{halfspacempp1}]
In view of  {\rm [\textbf{C}$_3(r)$]}, we only have to show that $\bal[r\widehat{H}(r)] = r\bv{g}$. Since $rH(r)$ is the tangent hyperplane to $L(r)$ at the point $r\bv{g},$   arguing as in the proof of Lemma \ref{uniquemmp}, we see that the sublevel set ${\rm cl}(\widehat{L}(r)^c)$ has a unique contact point $r\bv{g}$ with $rH(r)$.  It is clear that ${\rm int}(r\widehat{H}(r))$ is separated from  ${\rm cl}(\widehat{L}(r)^c)$ by the hyperplane $rH(r)$. Hence  $\Lambda(\bv{v}) > \Lambda(r\bv{g}),$   $\bv{v} \in {\rm int}(r\widehat{H}(r))$, which completes the proof of Lemma~\ref{halfspacempp1}.
\end{proof}

The   properties of the    half-space $\widehat{H}(r_G)$    stated in the next lemma will    play a key role in our argument.  It turns out that the crude asymptotics  of $\mathbb{P}\big(\eta( s\widehat{H}(r_G))<\infty\big)$ as $s \to \infty$ are the same as those  for  $\mathbb{P}\big(\eta( sG)<\infty\big)$  and, moreover, the MPTs and MPPs for the sets $G$ and $\widehat{H}(r_G)$ (and hence the GMPPs for them) are the same as well.

\begin{lemma}\label{DG}
Suppose the condition  {\rm [\textbf{C}$_3(r_G)$]} is met. Then
\begin{equation}\label{pair}
u_G = u_{\widehat{H}(r_G)} \quad {\it and} \quad \bal (r_G ) = \bal \big[ r_{\widehat{H}(r_G)}\widehat{H}(r_G)\big ] = r_G\bv{g},
\end{equation} so that $D(G) = D(\widehat{H}(r_G)),$  and $\big(u_G, u_G\bal (r_G )\big) = \big(u_{\widehat{H}(r_G)}, u_{\widehat{H}(r_G)}\bal \big[r_{\widehat{H}(r_G)} \widehat{H}(r_G) \big]\big )$ is the unique point at which the infimum on the RHS of \eqref{doubleinf} is attained for both $B = G$ and $B=\widehat{H}(r_G)$.

\end{lemma}

\begin{proof}[Proof of Lemma \ref{DG}]
First we   show that $(u_G, \bv{g})$ is the unique time-space point where the infimum on the RHS of \eqref{doubleinf} is attained when $B = G$. From [$\textbf{D}_3$] and {\rm [\textbf{C}$_3(r_G)$]},
\begin{equation}
\label{proofthm2.1-1}
D'(\bv{v})\big|_{\sbv{v} = \sbv{g}} = \Lambda' (\bal)\big|_{\sbal = r_G\sbv{g}} \equiv \bv{N}(r_G) \in Q^+.
\end{equation}
As the sublevel set $\widetilde{L}_1 := \{\bv{v} \in \mathbb{R}^d: D(\bv{v}) \leq D(\bv{g}) \}$ whose   boundary passes through   $\bv{g}$ is convex due to  [$\textbf{D}_1$], relation \eqref{proofthm2.1-1} means that $\widetilde{L}_1 \cap G=\{\bv{g}\}$. Therefore $D(G) = D(\bv{g})$ and $\bv{g}$ is the only point $\bv{v} \in G$ such that $D(G) = D(\bv{v})$.

By  [$\textbf{D}_2$], there is a unique point $t(\bv{g}) > 0$ such that $D(\bv{g}) = \Lambda(t(\bv{g})\bv{g})/t(\bv{g})$. Hence $(u_G = 1/t(\bv{g}), \bv{g})$ is the unique point at which the infimum on the RHS of \eqref{doubleinf} with $B = G$ is attained.

Now note that, in view of \eqref{proofthm2.1-1}, $H(r_G)$ is the tangent hyperplane to the level surface $\partial \widetilde{L}_1$ at the point $\bv{g}$. Arguing as in the proof of Lemma \ref{uniquemmp} and using the strict convexity of $D$ in a neighborhood of $\bv{g}$, which can be seen from [$\textbf{D}_3$] under condition {\rm [\textbf{C}$_3(r_G)$]}, we obtain that   $\widetilde{L}_1 \cap H(r_G)=\{\bv{g}\}$. Repeating (with obvious changes, replacing $G$ with $\widehat{H}(r_G)$) the argument in the first part of this proof, we see that $(1/t(\bv{g}), \bv{g})$ is also the unique point at which the RHS of \eqref{doubleinf} with $B = \widehat{H}(r_G)$ attains its minimum, so that $u_{\widehat{H}(r_G)} = 1/t(\bv{g}) = u_G$ and
\begin{equation*}
\bal \big[r_{\widehat{H}(r_G)}\widehat{H}(r_G)\big] = r_{\widehat{H}(r_G)}\bv{g} = r_G\bv{g} = \bal[r_GG] \equiv \bal(r_G) = r_G\bv{g}.
 \end{equation*} Lemma~\ref{DG} is proved.
\end{proof}

To prove the main Theorem~\ref{finalthm}, we will need a few further ancillary results. Recall   notations \eqref{MPPdef}, \eqref{31r}, \eqref{31x}  and $r:=s/n,$  and denote by
\begin{equation}
\label{bber}
 \bbe(r):=\bal[r\widehat{H}(r_G)]
\end{equation}
 the MPP of the set $r\widehat{H}(r_G)$. By Lemma~\ref{DG},
\begin{equation*}
\label{beta=alpha}
\bbe (r_G) =\bal (r_G)=r_G\bv{g}.
\end{equation*}

Denote by $\mathscr{G}$ the class of functions $\gamma : \mathbb{R}^+ \rightarrow \mathbb{R}^+$ such that $\gamma(s) = o(s)$ as $s \rightarrow \infty$. The next lemma describes the ``movement" of the MPPs for the half-spaces $r\widehat{H}(r_G) \equiv (s/n) \widehat{H}(r_G)$ for $n$-values in the $\gamma(s)$-neighborhood of~$s/r_G$.

 \begin{lemma}
 \label{nbconv}
 Let $\gamma \in \mathscr{G}$.   There exists  a constant  vector $\bv{\kappa}\in \mathbb{R}^d$ such that, as $s\to\infty$,  for $|n - s/r_G | \leq \gamma(s)$ one has
 \begin{equation*}
    n\bbe(s/n) - s\bv{g}  = (n - s/r_G) \bv{\kappa} + O(s^{-1}\gamma^2(s)).
 \end{equation*}
 \end{lemma}

\begin{proof}
Observe that
 \begin{equation}
 \label{chi}
\bv{\chi} := n\bbe(s/n) - s\bv{g}
 = n(\bbe(r) - r\bv{g})
= n[(\bbe(r) - \bbe(r_G)) + (r_G - r)\bv{g}].
\end{equation}
To evaluate the first term on the RHS, first recall that $\bbe(r) \in rH(r_G)$ according to~\eqref{aldB} and introduce the unit normal vector to $H(r_G)$ (cf.\ \eqref{normaln}):
\begin{equation}
\label{unitnormal}
\bv{\zeta} := \frac{\bv{N }(r_G)}{\|\bv{N}(r_G) \|}.
\end{equation}
Next  note that $rH(r_G) = r_GH(r_G) + \varepsilon\bv{\zeta}$, where
\begin{equation}
\label{varep}
\varepsilon := (r- r_G)\lag \bv{g}, \bv{\zeta} \rag =o(1),\qquad  s\to\infty,
\end{equation}
under the conditions of the lemma.
Choose an orthonormal system $\bv{e}_1, \ldots, \bv{e}_{d-1}$ of  vectors orthogonal to $\bv{\zeta}$ and let $J$ be the $ (d-1)\times d$--matrix having  these vectors as its rows. As  $\bbe (r) \in rH(r_G),$ this vector is of the form
\[
\bv{\beta}(r)=r_G\bv{g} + \varepsilon\bv{\zeta} +\sum_{i=1}^{d-1}h_i\bv{e}_i
 = r_G\bv{g} + \varepsilon\bv{\zeta} + \bv{h}J,\quad \bv{h}:=(h_1,\ldots, h_d)\in \mathbb{R}^{d-1}.
\]
As $\bbe (r)$ is the MPP for $r\widehat{H}(r_G)$, it  is the unique point of that form which is orthogonal to $H(r_G)$  or, which is the same, orthogonal to all $\bv{e}_j, $ $j=1,\ldots, d-1:$
\begin{equation}\label{orthg}
 \bla (r_G\bv{g} + \varepsilon\bv{\zeta} + \bv{h}J) J^T=\bv{0}.
\end{equation}
Next, assuming that $\|\bv{h}\| = o(1)$, we use condition {\rm [\textbf{C}$_3(r_G)$]},   the multivariate Taylor's formula and~\eqref{bla(bal)} to write
\begin{equation*}
\bla(r_G\bv{g} +\varepsilon \bv{\zeta} + \bv{h}J ) = \bla(r_G\bv{g}) + (\varepsilon \bv{\zeta} + \bv{h}J)\Lambda''(r_G\bv{g}) + O(\varepsilon^2  +\|\bv{h}\|^2).
\end{equation*}
Substituting this into~\eqref{orthg}, noting that $\bla(r_G\bv{g})J^T=\bv{0}$ and  setting $A:=\Lambda'' (r_G \bv{g})$  for brevity, we get
\[
 ( \varepsilon\bv{\zeta} + \bv{h}J) A J^T+O(\varepsilon^2 +\|\bv{h}\|^2)=\bv{0}.
\]
The remainder term here is a continuous function of $\bv{h}$, whereas $J AJ^T$ is a positive-definite matrix since $A$ is. So we conclude that there exists a (unique, as we already know) solution  to the above equation equal to $\bv{h}= -\varepsilon\bv{\zeta} A J^T (J  A J^T)^{-1} + O(\varepsilon^2  ).$ Hence
\begin{equation}
\label{beta_dist}
\bbe(r) - \bbe(r_G)\equiv
\bbe(r) - r_G\bv{g}
=
\varepsilon \big( \bv{\zeta} - \bv{\zeta} A J^T (J A J^T)^{-1} J \big) + O(\varepsilon^2).
\end{equation}
It follows from~\eqref{chi}, \eqref{varep} and~\eqref{beta_dist} that
\begin{align}
  \bbe(r) -r\bv{g}   &=  \bbe(r) - \bbe(r_G)  + (r_G - r)  \bv{g}
  \notag
\\
&= \varepsilon \big( \bv{\zeta} - \bv{\zeta} A J^T (J  A J^T)^{-1} J \big) + O(\varepsilon^2)+ (r_G - r)  \bv{g}
\notag
 \\
&=
( r_G-r) \bv{  \kappa }/r_G +O\big((r_G-r) ^2\big),
\label{beta-rg}
\end{align}
where $
\bv{  \kappa }:= r_G\big[\big( \bv{\zeta} A J^T (J  A J^T)^{-1} J  -\bv{\zeta} \big) \lag\bv{g},\bv{\zeta}\rag  + \bv{g}\big] .$
As $n(r_G - r) = r_G(n- s/r_G)$ and $n( r_G-r)^2 = n^{-1}r_G^2(n-s/r_G)^2 = O(s^{-1}\gamma^2(s))$, the lemma is proved.
\end{proof}

For $\bal \in \Omega_{\Lambda}$, recall~\eqref{12a} and introduce the following two functions of $\bv{z} \in \mathbb{R}^d$:
 \begin{equation*}\label{piftn}
p(\bv{z}) := \mathbb{P}\big( \eta \big( {\rm cl}(Q^+) - \bv{z} \big) < \infty \big),
\end{equation*}
so that $p(\bv{z}) = 1$ for $\bv{z} \in {\rm cl}(Q^+),$ and
\begin{align}
 \label{qftn}
 q_{\sbal}(\bv{z}) &:= \mathbb{P}\Big(\inf_{n \geq 1} \lag \bla(\bal(r_G)), \bv{S}^{(\sbal)}(n)\rag \geq \lag \bla(\bal(r_G)), \bv{z} \rag \Big)
 \end{align}
(cf.\ pp.~253--254 in \cite{mainpaper}; in fact, $q_{\sbal}$ was defined there as an integral involving the RHS of~\eqref{qftn}, but on close inspection it is easy to see that it is actually the same as~\eqref{qftn}).  For a Borel subset $W \subset \widehat{H}_0(r_G)$, a    $\bv{w}  \in H_0(r_G)$ and $r > 0$ such that $\bbe(r) \in \Omega_{\Lambda}$, set
\begin{equation*}
 E(r, \bv{w}, W) := \int_{W} e^{-\lag \sbla(\sbbe(r)), \sbv{v} \rag} p(\bv{w} + \bv{v})q_{\sbbe(r)}(\bv{v})d\bv{v} < \infty,
\end{equation*}
the last inequality being a consequence of the bound~\eqref{new50} below for   $p$ and the fact that $\bv{\lambda}(\bbe(r)) \perp H_0(r_G).$
Finally, denote by $\mathcal{P}$   the orthogonal projection onto~$H_0(r_G).$

The next theorem is a key step in implementing our approach based on auxiliary half-spaces. If the RW $\bv{S}$ hits $sG$, then it   inevitably hits the ``best half-space approximation" $s\widehat{H}(r_G)\supset sG$ to it (in the sense that both sets have the same crude hitting probabilities asymptotics). In Theorem~\ref{mainthm1}, we compute the probability of hitting $sG$ ``localizing" in both time and space when and where the RW first hits~$s\widehat{H}(r_G)$.

\begin{theorem} \label{mainthm1}
Put $\bv{w} := n\bbe(r) -s\bv{g} + \bv{x}$.  There exists a sequence $\delta_n\to 0$  such that,  for  any fixed $\Delta_0 > 0$, $M_0 \in (0, \infty)$, and $\gamma \in \mathscr{G}$, one has, as $s \to \infty,$
 \begin{multline}
 \label{firstresult11}
 \mathbb{P}\big( \eta(sG) < \infty,  \eta \big(s\widehat{H}(r_G)\big) = n ,\bv{S}(n) \in n\bbe(r) + \bv{x} + \Delta[\bv{y}) \big)  \\
 =
  \frac{ \exp\{-n\Lambda(\bbe(r)) - \frac{1}{2n}\bv{x}\Lambda''(\bbe(r))\bv{x}^T + O(\|\bx \|^3n^{-2}) \}  }{(2\pi n)^{d/2} \sigma(\bbe(r))}
\\
\times \big[E(r, \bv{w}, \Delta[\bv{y}))(1+o(1))  + o\big(\Delta^d\exp\{-c_1\|\mathcal{P}(\bv{w}+\bv{y}) \| - c_2\lag \bv{\zeta}, \bv{y} \rag \} \big)  \big]
  \end{multline}
 uniformly in the range of the variables $n$, $\bv{x} \in H_0(r_G)$ and $\bv{y}$ specified by:
\begin{align*}
 &\big|n - s/r_G\big|  \leq \gamma (s),
 \quad \Delta \in [\delta_n , \Delta_0],
  \\
 &\|\bv{x} \| \leq \gamma(s), \quad \| \bv{y}\| < M_0, \quad \bx + \Delta[\bv{y}) \subset \widehat{H}_0(r_G).
\end{align*}
\end{theorem}

\begin{remark}
{\rm
The point of separating the variables $\bv{x}$ and $\bv{y}$ in the statement of this theorem is that it will be convenient in the next step (Corollary~\ref{coro1})  of the proof of our main result. At that step, we will obtain a representation similar to~\eqref{firstresult11} where instead of the ``small'' cube $\Delta[\bv{y})$ we will have a half-cylinder with a ``small'' base $\Delta^*[\bv{x}) \subset H_0(r_G)$ and generatrix parallel to $\bv{\zeta}$ (to be achieved by ``integrating'' the asymptotics from~\eqref{firstresult11} with respect to $\bv{y}$).

}
\end{remark}

\begin{proof}[Proof of Theorem~\ref{mainthm1}]
Assume for simplicity that $d = 2$ (we will explain at the end of the proof  how the argument changes in the   case $d\ge 3$). Put $\Delta_m := \Delta m^{-1},$ where  $m=m(n)\to \infty$ as $n\to\infty$   slowly enough (the choice of $m$ is discussed below).  For $\bv{y} = (y_1, y_2), $ set
 \begin{equation*}\label{zij}
 \bv{z}^{i,j} := (y_1 + (i-1) \Delta_m, y_2 + (j-1) \Delta_m  ), \quad i,j \ge 1,
\end{equation*}
and partition the square $\Delta[\bv{y})$ into $m^2$ sub-squares $\Delta_m[\bv{z}^{i,j})$:
$
\Delta[\bv{y}) = \bigcup_{1\leq i,j \leq m} \Delta_m[\bv{z}^{i,j}).
$
Clearly, setting $\bv{x}' := n\bbe(r) + \bv{x} \equiv \bv{w} + s\bv{g},$ we have
\begin{align}
P :&=
\mathbb{P}\big( \eta(sG) < \infty, \eta(s\widehat{H}(r_G)) = n, \bv{S}(n) \in \bx' + \Delta[\bv{y}) \big)
\notag
 \\
  &= \sum_{1\leq i,j \leq m}  \mathbb{P}\big( \eta(sG) < \infty, \eta(s\widehat{H}(r_G)) = n,  \bv{S}(n) \in \bx' + \Delta_m[\bv{z}^{i,j})  \big).
\label{ltp1}
\end{align}
Due to the Markov property, the $(i,j)$-th term  in the sum on the RHS of \eqref{ltp1} equals
\begin{align*}
\label{longint}
 \int_{\Delta_m[\sbv{z}^{i,j})}  &\mathbb{P}\big( \eta(sG) < \infty, \eta(s\widehat{H}(r_G)) = n,  \bv{S}(n) \in \bx' + d\bv{v}  \big)
  \nonumber \\
  & =\int_{\Delta_m[\sbv{z}^{i,j})}  \mathbb{P}\big( \eta(sG) < \infty \, \big| \, \eta(s\widehat{H}(r_G)) = n,  \bv{S}(n) = \bx' + \bv{v}  \big) \nonumber \\
&\hphantom{ \int_{\Delta_m[\sbv{z}^{i,j})}}\
\times
 \mathbb{P}\big( \eta(s\widehat{H}(r_G)) = n,  \bv{S}(n) \in \bx' + d\bv{v}  \big)
 \nonumber \\
& =\int_{\Delta_m[\sbv{z}^{i,j})}  p(\bv{w} + \bv{v}) \mathbb{P}\big( \eta(s\widehat{H}(r_G)) = n,  \bv{S}(n) \in \bx' + d\bv{v}  \big)  = : I_{i,j}.
\end{align*}

Now introduce the time-reversed RW
$
 \widetilde{\bv{S}}(k) := \bxi(n) + \bxi(n-1) + \cdots + \bxi(n-k+1),$ $ 1 \leq k \leq n. $
Note that $\eta(s\widehat{H}(r_G))$ is the first time the univariate RW $\{\lag {\bv{S}}(k), \bv{\zeta} \rag\}_{k\ge 0}$ hits the level  $\lag  \bv{x}', \bv{\zeta} \rag$ and that  $ \lag \bv{w}, \bv{\zeta}\rag =0,$  $ \lag \bv{v}, \bv{\zeta}\rag >0$ for $\bv{v}\in \Delta_m[\sbv{z}^{i,j}),$ so that
\begin{equation*}
\{\eta(s\widehat{H}(r_G)) = n, \SU{n} = \bv{x}' + \bv{v} \} = \Big\{\min_{1 \leq k \leq n} \lag \widetilde{\bv{S}}(k), \bv{\zeta} \rag > \lag \bv{v}, \bv{\zeta} \rag, \SU{n} = \bv{x}' + \bv{v} \Big\}.
\end{equation*}
Further, the  function $p(\bv{z})$ is non-decreasing along any ray with a directional vector  $\bv{v} \in {\rm cl}(Q^+)$:
as ${\rm cl}(Q^+) -\bv{z} \subset {\rm cl}(Q^+) -\bv{z} - \bv{v}$ for such $\bv{v}$, one has
\begin{equation}\label{ineqpi2}
\begin{aligned}
p(\bv{z}+\bv{v}) = \mathbb{P}( \eta({\rm cl}(Q^+) - \bv{z} - \bv{v}) < \infty )
\geq \mathbb{P}( \eta({\rm cl}(Q^+) - \bv{z}) < \infty ) = p(\bv{z}). 
\end{aligned}
\end{equation}
Therefore,
\begin{equation}
\label{minmaxpi}
\min_{\sbv{v} \in \Delta_m[\sbv{z}^{i,j}) }p(\bv{v}) =  p(\bv{z}^{i,j}), \qquad \max_{\sbv{v} \in \Delta_m[\sbv{z}^{i,j}) }p(\bv{v}) = p(\bv{z}^{i+1,j+1})
\end{equation}
 and, as clearly $
\lag \bv{z}^{i,j}, \bv{\zeta}\rag \leq \lag \bv{v}, \bv{\zeta} \rag$ for $\bv{v} \in \Delta_m[\bv{z}^{i,j}),$
we obtain that
\begin{align}
\label{52a}
I_{i,j} &\leq \int_{\Delta_m[\sbv{z}^{i,j})} p(\bv{w} + \bv{z}^{i+1, j+1})\mathbb{P}\Big(\min_{1 \leq k \leq n} \lag \widetilde{\bv{S}}(k), \bv{\zeta} \rag > \lag \bv{z}^{i,j}, \bv{\zeta}\rag, \SU{n} = \bv{x}' + d\bv{v} \Big) \notag
\\
&= p(\bv{w} + \bv{z}^{i+1, j+1})  \mathbb{P}\Big(\min_{1 \leq k \leq n} \lag \widetilde{\bv{S}}(k), \bv{\zeta} \rag > \lag  \bv{z}^{i,j}, \bv{\zeta}\rag, \SU{n} \in \bv{x}' + \Delta_m[\bv{z}^{i,j}) \Big)  \notag
 \\
&= p(\bv{w} + \bv{z}^{i+1, j+1})  \mathbb{P}\Big(\min_{1 \leq k \leq n} \lag \widetilde{\bv{S}}(k), \bv{\zeta} \rag > \lag \bv{z}^{i,j}, \bv{\zeta} \rag \big| \SU{n} \in \bv{x}' + \Delta_m[\bv{z}^{i,j}) \Big)  \notag
\\
&\hspace*{60 mm}\times \mathbb{P}\big(\SU{n} \in \bv{x}' + \Delta_m[\bv{z}^{i,j}) \big).
\end{align}
Asymptotic representations for the second and third factors on the RHS can be obtained, respectively, from Theorems~10 and~9 in~\cite{mainpaper}.   The assumptions of these theorems in~\cite{mainpaper} include Cram\'er's strong non-lattice condition $(C_2)$ on the characteristic function of~$\bxi$, but that condition is actually unnecessary provided that $\bxi$ is just non-lattice and  the ``small cube'' edge is only allowed to decay slowly enough (the key tool for such an extension is  the integro-local Stone's theorem, for more detail  see e.g.~\cite{paper10}). Under such weakened conditions, the assertions of Theorems~10 and~9 in~\cite{mainpaper} will still hold   uniformly in the small cube edge lengths in the interval $ [\delta_n',\Delta_0]$ for some sequence $\delta_n'\to 0$.

Now we will choose $m=m(n)\to \infty$ such that $\delta_n:=\delta_n'm\to 0$ as $n\to\infty$
	Since $\bx'/n = \bbe(r) + o(1),$ by the modified version of Theorem~10 in \cite{mainpaper}, for the second factor on the RHS of~\eqref{52a} we   have
	\begin{equation*}
	\mathbb{P}\Big(\min_{1 \leq k \leq n} \lag \widetilde{\bv{S}}(k), \bv{\zeta} \rag > \lag  \bv{z}^{i,j}, \bv{\zeta}
	\rag  \Big| \SU{n} = \bv{x}' + \Delta_m[\bv{z}^{i,j}) \Big) = q_{\sbbe(r)}(\bv{z}^{i,j})(1+o(1))
	\end{equation*}
(cf.\ p.\,264 in \cite{mainpaper}), whereas by the modified version of     Theorem~9 in \cite{mainpaper} (which, roughly speaking, is just a combination of Stone's integro-local theorem with Cram\'er's change of measure, a multi-variate version of Theorem~9.3.1 in~\cite{mainbook2}) for the third factor  on the RHS of~\eqref{52a}  one has the relation
\begin{multline*}
	\label{ilthm}
	\mathbb{P}\big(\SU{n} \in \bv{x}' + \Delta_m[\bv{z}^{i,j}) \big) = \frac{\Delta_m^2(1+o(1))}{2\pi n \sigma((\bv{x}' + \bv{z}^{i,j})/n)} \exp \{-n\Lambda(\bbe(r) + (\bv{x} + \bv{z}^{i,j})/n ) \}.
	\end{multline*}
Now, expanding the rate function in the exponential on the RHS about the point $\bbe(r)$ and using~\eqref{bla(bal)}, we obtain the following representation for the probability on the LHS:
\[
 \frac{\Delta^2_m(1+o(1))}{2\pi n\sigma(\bbe(r))}\exp \Big\{\!\!- \!n\Lambda(\bbe(r))\! -\! \lag \bla(\bbe(r)), \bv{z}^{i,j} \rag \!-\! \frac{1}{2n}\bv{x} \Lambda''(\bbe(r))\bv{x}^T + \theta_{i,j} \Big\},
\]
where the remainders $o(1)$ and $\theta_{i,j} = O(\|\bx^3 \|n^{-2})$ are both uniform in $\Delta \in [\delta_n, \Delta_0]$ and $\bv{z}^{i,j} \in \mathbb{R}^d$, $\bv{x} \in H_0(r_G)$ such that $\|\bx \| \leq \gamma(s)$, $\|\bv{z}^{i,j} \| < M_0$, $\bv{x} + \Delta[\bv{z}^{i,j}) \subset s\widehat{H}(r_G)$.
Here we used the Taylor expansion of $\Lambda$ at $\bbe(r),$  relation~\eqref{bla(bal)}  and that $\lag \bla(\bbe(r)), \bx \rag = 0$ for $\bx \in H_0(r_G).$ Combining the above representations for the factors on the RHS of~\eqref{52a} yields an upper bound for $I_{i,j}$.

In the same way, but  using now the first relation in \eqref{minmaxpi} and the observation that $
\lag \bv{z}^{i+1,j+1}, \bv{\zeta} \rag \geq \lag \bv{v}, \bv{\zeta} \rag,   $   $  \bv{v} \in \Delta_m[\bv{z}^{i,j}),$
we obtain a lower bound for $I_{i,j}$ of the same form as the upper one, but involving $p(\bv{w} + \bv{z}^{i,j})$ and $q_{\sbbe(r)}(\bv{z}^{i+1, j+1})$  on its RHS.

Summing up the obtained upper and lower bounds for $I_{i,j}$, $1 \leq i, j \leq m$, we see from \eqref{ltp1} that
\begin{align*}
\Delta_m^2 & \sum_{1 \leq i, j \leq m} p(\bv{w}   + \bv{z}^{i,j})q_{\sbbe(r)}(\bv{z}^{i+1, j+1})e^{-\lag \sbla(\sbbe(r)), \sbv{z}^{i, j} \rag}(1+o(1))
  \\
& \leq J:= 2\pi n \sigma(\bbe(r))\exp \Big\{n\Lambda(\bbe(r)) + \frac{1}{2n}\bv{x}\Lambda''(\bbe(r))\bv{x}^T - \theta \Big\} P
\\
& \leq \Delta_m^2 \sum_{1 \leq i, j \leq m} p(\bv{w} + \bv{z}^{i+1,j+1})q_{\sbbe(r)}(\bv{z}^{i, j})e^{-\lag \sbla(\sbbe(r)), \sbv{z}^{i, j} \rag}(1+o(1)),
\end{align*}
where $\theta = O(\|\bx^3 \|/n^2)$.  As $\|\bv{z}^{i,j}- \bv{z}^{i+1,j+1}\|=2^{1/2}\Delta/m\to 0,$  we can  now  replace $\lag \bla(\bbe(r)), \bv{z}^{i,j} \rag$ in the lower bound  with $\lag \bla(\bbe(r)), \bv{z}^{i+1,j+1} \rag,$ yielding
\begin{multline*}
\Delta_m^2 \sum_{1 \leq i, j \leq m} p(\bv{w} + \bv{z}^{i,j})q_{\sbbe(r)}(\bv{z}^{i+1, j+1})e^{- \lag \sbla(\sbbe(r)), \sbv{z}^{i+1,j+1} \rag}(1+o(1))  \\
\leq J
\leq \Delta_m^2 \sum_{1 \leq i, j \leq m} p(\bv{w} + \bv{z}^{i+1,j+1})q_{\sbbe(r)}(\bv{z}^{i, j})e^{- \lag \sbla(\sbbe(r)), \sbv{z}^{i,j} \rag}(1+o(1)).
\end{multline*}
Observe that the LHS (RHS) in the above formula is, up to the factor $(1+o(1))$, the lower (upper) Darboux sum for the function
\begin{equation}
\label{48a}
p(\bv{w} + \bv{z})q_{\sbbe(r)}(\bv{z})e^{- \lag \sbla(\sbbe(r)), \sbv{z} \rag}, \qquad \bv{z} \in \Delta[\bv{y}).
\end{equation}
It is not hard to see that the difference between the sums vanishes uniformly as $s \to \infty$, and so they both tend to the Riemann integral $E(r, \bv{w}, \Delta[\bv{y}))$ of that function over $\Delta[\bv{y})$.

Indeed, setting, for a function $h(\bv{z})$, $\bv{z} \in \mathbb{R}^2$,
\begin{equation*}
\overline{h}^{i,j} := h(\bv{z}^{i+1, j+1}), \quad \underline{h}^{i,j} := h(\bv{z}^{i,j}), \quad i, j \geq 1
\end{equation*}
(the values of $h$  at the top-right and left-bottom vertices of the sub-squares $\Delta_m[\bv{z}^{i,j})$, respectively)  and letting $
f(\bv{z}) := p(\bv{w}+\bv{z}), $   $
g(\bv{z}) := q_{\sbbe(r)}(\bv{z})e^{-\lag \sbla(\sbbe(r)), \sbv{z} \rag},$
the difference between the upper and lower Darboux sums for~\eqref{48a} on $\Delta[\bv{y})$ can be written, suppressing the superscripts $i,j$ in all the factors, as
\begin{equation*}
\delta := \Delta_m^2\sum_{1 \leq i,j \leq m}\big( \overline{f} \underline{g} - \underline{f} \overline{g}\big).
\end{equation*}
Using monotonicity of both $f(\bv{z})$ (see~\eqref{ineqpi2}) and the exponential factor $e^{-\lag \sbla(\sbbe(r)), \sbv{z} \rag}$ along directions from $Q^+$, we can bound the value of the sum here as follows:
\begin{align}
\label{dsum}
\sum_{1 \leq i,j \leq m}\big( \overline{f}\underline{g} - \underline{f}\overline{g}\big)
& =
\sum_{1 \leq i,j \leq m} \big(\overline{f}-\underline{f} \big)\underline{g} + \sum_{1 \leq i,j \leq m} \underline{f} \big(\underline{g}-\overline{g} \big)
\notag
\\
& \leq
e^{-\lag \sbla(\sbbe(r)), \sbv{y} \rag}\sum_{1 \leq i,j \leq m} \big(\overline{f}-\underline{f} \big) + f(\bv{z}^{m,m})\sum_{1 \leq i,j \leq m} \big(\underline{g}-\overline{g} \big).
\end{align}
Since $\overline{f}^{i,j} = \underline{f}^{i+1,j+1}$, $1 \leq i, j \leq m-1$, using the telescoping argument we see that the first sum on the RHS of~\eqref{dsum} equals
\begin{align*}
 \sum_{2 \leq i \leq m+1} f(\bv{z}^{i,m+1}) &  -  \sum_{1 \leq i \leq m} f(\bv{z}^{i,1})
 \\
 & + \sum_{2 \leq j \leq m} f(\bv{z}^{m+1,j}) -\sum_{2 \leq j \leq m} f(\bv{z}^{1,j})  \leq 2mf(\bv{z}^{m+1,m+1}),
\end{align*}
whereas  the second sum on the RHS of~\eqref{dsum}, using the same argument, is seen to be bounded from above by $
2mg(\bv{y}) \leq 2me^{-\lag \sbla(\sbbe(r)), \sbv{y} \rag}.
$
Summarizing, we obtain that
\begin{align}
\label{55a}
\delta \leq 4\Delta^2 m^{-1}e^{-\lag \sbla(\sbbe(r)), \sbv{y} \rag} f(\bv{z}^{m+1,m+1}),
\end{align}
where $f(\bv{z}^{m+1, m+1}) = p(\bv{w} + \bv{y} + (\Delta, \Delta)).$

To bound the last quantity, we will derive a bound for the function $p(\bv{u})$ in the general case $d \geq 2$.
It follows from the condition that $\lag \mathbb{E}\bxi, \bv{\zeta} \rag < 0$ (part of  {\rm [\textbf{C}$_3(r_G)$]}) that there exists a
\begin{equation}
\label{Cone_C}
\begin{split}
&\mbox{closed round cone $C \supset Q^+$ with the axis direction $\bv{\zeta}$, apex at $\bv{0}$ }\\
&\mbox{and the opening angle $\pi - 2\phi$ with $\phi > 0$ such that $- \mathbb{E}\bxi \in C$.}
\end{split}
\end{equation}
Clearly, $C \subset \widehat{H}_0(r_G)$. For any $\bv{u} \in \widehat{H}_0(r_G)\backslash C$, denote by $
\bv{u}' := \argminA_{\sbv{v} \in C}\|\bv{u} - \bv{v} \|
$
the nearest to $\bv{u}$ point of $C$ and let
\begin{equation*}
\bv{\varkappa}(\bv{u}) := \frac{\bv{u}' - \bv{u}}{\|\bv{u}' - \bv{u} \|}
\end{equation*}
be the inner normal to $\partial C$ at that point. Denote by
$
\widehat{T}(\bv{u}) := \{\bv{v} \in \mathbb{R}^d: \lag \bv{v}, \varkappa(\bv{u}) \rag \ge 0 \}
$
the half-space containing $C$ and bounded by the tangent to $\partial C$ hyperplane passing through the point $\bv{u}'$ (and the origin). Clearly,
\begin{align*}
p(\bv{u}) &\leq \mathbb{P}\big(\eta(C-\bv{u}) < \infty \big)
\\
& \leq \mathbb{P}\big(\eta(\widehat{T}(\bv{u}) - \bv{u}) <  \infty \big)
\\
& \leq \mathbb{P}\Big(\sup_{n \geq 1} \lag \bv{S}(n), \varkappa(\bv{u}) \rag \ge \|\bv{u}' - \bv{u} \| \Big)
\\
& = \mathbb{P}\Big(\sup_{n \geq 1}S_{\sbv{u}}(n) \geq (\| \mathcal{P}(\bv{u}) \|\tan \phi - \lag \bv{u}, \bv{\zeta} \rag )\sin \phi   \Big),
\end{align*}
where $S_{\sbv{u}}(n) := \lag \bv{S}(n), \varkappa(\bv{u}) \rag \equiv \sum_{k = 1}^{n} \lag \bxi(k), \varkappa(\bv{u}) \rag$ is a univariate RW with the negative drift: $
\mathbb{E} \lag \bxi, \varkappa(\bv{u})\rag = - \lag -\mathbb{E}\bxi, \varkappa(\bv{u}) \rag < 0
$
since $- \mathbb{E} \bxi \subset C \subset \widehat{T}(\bv{u})$ and $\varkappa(\bv{u})$ is the inner normal vector to $\partial \, \widehat{T}(\bv{u}),$ so that $\lag - \mathbb{E}\bxi, \varkappa(\bv{u}) \rag > 0.$ Therefore
\begin{equation}
\label{new50}
p(\bv{u}) \leq e^{-\nu(\varkappa(\sbv{u}))(\| \mathcal{P}(\sbv{u}) \|\tan \phi - \lag \sbv{u}, \sbv{\zeta} \rag )\sin \phi},
\end{equation}
where $
\nu(\varkappa(\bv{u})) := \sup\{\nu \in \mathbb{R}: \mathbb{E} e^{\nu \lag \sbv{\xi}, \varkappa(\sbv{u}) \rag} \leq 1  \} > 0
$
(see p.\ 81 in~\cite{ASM}). That $\nu(\varkappa(\bv{u})) > 0$ follows from condition {\rm [\textbf{C}$_3(r_G)$]} and the fact that $\phi > 0$ can be chosen arbitrary small thus making all the vectors  $\varkappa(\bv{u})$ with $ \bv{u} \in \widehat{H}_0(r_G)\backslash C $ arbitrary close to $\bv{\zeta} \equiv \bla(\bal(r_G))/\|\bla(\bal(r_G)) \|$ with $\bla(\bal(r_G)) \in \Theta_{\psi}$. This also implies that
\begin{equation*}
\nu_0 := \inf_{\sbv{u} \in \widehat{H}_0(r_G)\backslash C} \nu(\varkappa(\bv{u})) > 0,
\end{equation*}
which, together with~\eqref{55a} and~\eqref{new50}, yields the bound
\begin{align*}
\delta &\leq c\Delta^2m^{-1}\exp\{-\lag \bla(\bbe(r)), \bv{y} \rag - \nu_0(\| \mathcal{P}(\bv{w} + \bv{y}) \|\tan \phi - \nu_0\lag \bv{\zeta}, \bv{y} \rag ) \sin \phi     \}
\\
&\leq c\Delta^2m^{-1}\exp\{ -c_1\|\mathcal{P}(\bv{w} + \bv{y}) \| - c_2\lag \bv{\zeta}, \bv{y} \rag \}
\end{align*}
for small enough $c_1, c_2 > 0$ (as $\bla(\bbe(r)) = h\bv{\zeta}$ for $h$ bounded away from zero and $\phi$ can be chosen arbitrary small). Therefore,
\begin{equation*}
J = E(r, \bv{w}, \Delta[\bv{y}))(1+o(1)) + o\big(\Delta^2e^{-\lag \sbbe(r), \sbv{y} \rag -c_1\|\mathcal{P}(\sbv{w}+\sbv{y}) \| - c_2\lag \sbv{\zeta}, \sbv{y} \rag}\big)
\end{equation*}
uniformly in the specified range.  This completes the proof in the case $d=2$.

For $d\ge 3,$   we partition $\Delta[ \boldsymbol{y} )\subset\mathbb{R}^d$ into $m^d$ small cubes (instead of  $m^2$ small squares, as in the case $d=2$). After that, all the computations are done in the same way as above  (including~\eqref{minmaxpi}, where the min and max of $p$ are now attained at the opposite vertices of the small cubes), except for the ``telescoping argument" following~\eqref{dsum}. Instead of the sums over the nodes   on the edges of the square $\Delta[ \boldsymbol{y} ),$ we end up now with sums over the nodes on the faces of the cube  $\Delta[ \boldsymbol{y} ),$ yielding a factor $m^{d-1}$ instead of $m$. But as we then divide the result by $m^d$ (instead of $m^2$, which was the case when $d=2$), we end up with the same desired final result. Theorem~\ref{mainthm1} is proved.
\end{proof}

Next we will use Theorem~\ref{mainthm1}, ``integrating" representation~\eqref{firstresult11} to compute the probability of ever hitting $sG$ localizing only the time when  $\bv{S}$ first hits~$s\widehat{H}(r_G)$ and the projection onto $H_0(r_G)$ of the point where $\bv{S}$   enters that set. This  result will be used in the key step in the proof of Theorem~\ref{finalthm}, when evaluating the contribution of the main term~$P_3$ (to be defined in~\eqref{PPP}).

Fix a cartesian coordinate system in the hyperplane $H_0(r_G)$ and, for $\bv{v} \in H_0(r_G)$ and $\Delta > 0$, denote by $\Delta^*[\bv{v})$ the $(d-1)$-dimensional cube in $H_0(r_G)$ with edges parallel to the axes in the chosen coordinate system, the ``left--bottom'' vertex at $\bv{v}$ and the edge length~$\Delta$ (cf.~\eqref{deltay}). Denote by
\begin{equation*}
W(\Delta^*[\bv{v})) := \bigcup_{t \geq 0}\{\Delta^*[\bv{v}) + t\bv{\zeta} \}
\end{equation*}
the half-cylinder with the  base $\Delta^*[\bv{v})$ and generatrix parallel to the unit normal $\bv{\zeta}$ to $H_0(r_G)$. Recall notation $\bv{w} = n\bbe(r) - s\bv{g} + \bv{x}$ from Theorem~\ref{mainthm1} and set
 \begin{equation}
 \label{XIXI}
 \Xi(s,n):= \frac{e^{-n\Lambda(\sbbe(r)) }     }{(2\pi n)^{d/2} \sigma(\bbe(r))}, \qquad {\rm where} \quad r = \frac{s}{n}.
 \end{equation}
Following Remarks $1$ and $3$ from \cite{mainpaper}, one can ``tile'' the half-cylinder $W(\Delta^*[\bv{0}))$ with ``small'' cubes $\Delta'[\bv{y})$ with $\Delta' \to 0$ and then sum up the representations for those small cubes given by Theorem~\ref{mainthm1}  thus ``integrating'' these local representations to obtain the following result.

\begin{corollary}
\label{coro1}
There exists a sequence $\delta_n^*\to 0$ as $n\to\infty$ such that, for any fixed $\Delta_0 >0$ and $\gamma \in \mathscr{G}$, one has, as $ s \to \infty$,
\begin{align}
\label{firstresult2}
 \mathbb{P}\big( \eta(sG) < \infty,  \eta \big(s\widehat{H}(r_G)\big) &= n ,\bv{S}(n) \in n\bbe(r) + \bv{x} + W(\Delta^*[\bv{0})) \big) \notag
 \\
 &=
 \Xi (s,n)\exp\Big\{-\frac{1}{2n}\bx \Lambda''(\bbe(r))\bx^T + O\Big(\frac{\|\bx \|^3}{n^2} \Big) \Big\} \notag
\\
&\hspace{10 mm} \times \big[E\big(r, \bv{w}, W(\Delta^*[\bv{0}))  \big)(1  + o(1)) + R \big],
\end{align}
where $
R = o\Big( \int_{\Delta^*[\sbv{0})} e^{-c_1\|\sbv{w} \|} d\mu(\bv{w})\Big),
$
$\mu$ being the $(d-1)$-dimensional volume measure on $H_0(r_G)$, the $o(\cdot)$-term being uniform in  $\bv{x} \in H_0(r_G)$ and $n \geq 1$ such that $\|\bx \| \leq \gamma(s) $,  $\big|n - s/r_G\big|  \leq \gamma(s) $ and $\Delta \in [\delta_n^* , \Delta_0]$.
\end{corollary}

We just note here that the bound for $R$ is obtained by choosing $\bv{y} \perp H_0(r_G)$ in Theorem~\ref{mainthm1} and integrating along the direction of $\bv{\zeta}$.

Now we are ready to proceed to proving the main result of the paper.

\begin{proof}[{\it Proof of Theorem~\ref{finalthm}}]
First we will partition the half-space  $s\widehat{H}(r_G)\supset s G$ into several  subsets and, for each of them,   evaluate the probability of ever hitting $sG$  when the RW first hits  $s\widehat{H}(r_G)$ in the respective partition element.    The ways we will be doing these computations  will be different for different elements of the partition.

We will now assume that   $d=2$ as in this case it is easier to explain how we do the evaluation. The construction to be used when $d\ge 3$ is described later, just after~\eqref{no2}.

Let $\bv{e} = (e_1, e_2) := (\zeta_2, -\zeta_1)$ be the unit vector orthogonal to $\bv{\zeta}$ such that $e_1 > 0$.
For $M \geq 1$ (to be chosen later), put $\bv{a}_{\pm} := s\bv{g} \pm (M\ln s)\bv{e}$ and  consider the sets
 \begin{align*}
V_{+}  := \{\bv{v} \in s\widehat{H}(r_G): \lag \bv{v}, \bv{e}\rag \ge   \lag \bv{a}_{+}, \bv{e} \rag \},
 \quad
 V_{-}  := \{\bv{v} \in s\widehat{H}(r_G): \lag \bv{v}, \bv{e}\rag <  - \lag \bv{a}_{-}, \bv{e} \rag \}.
 \end{align*}
Next we will  split each of the sets $V_{\pm}$ into two parts. We need to consider two alternative situations, depending  on whether   $\mathbb{E}\bv{\xi}$ is in $ - Q^+$   or not.

{\em Case $\mathbb{E} \bxi \in - Q^+.$} In that case, we put  (see Fig.\ 2)
\begin{align*}
V_{1+} := V_+ \cap \{ \bv{v}: v_2 \leq sg_2 - \mbox{$ \frac12$} (M\ln s)|e_2|  \},  \quad
V_{1-}  := V_{-} \cap \{\bv{v}: v_1 \leq sg_1 - \mbox{$ \frac12$}  (M\ln s)e_1 \}
\end{align*}
and set
\begin{equation}
\begin{aligned}
\label{no1}
V_{2-} &:= V_{-}\backslash V_{1-},
\qquad
 V_{2+} :=  V_{+} \backslash V_{1+},\\
V_1 := V_{1+} \cup &V_{1-}, \quad V_2 := V_{2+} \cup V_{2-}, \quad V_3 := s\widehat{H}(r_G) \backslash (V_{-} \cup V_{+}).
\end{aligned}
\end{equation}

\begin{figure}[ht]	
	\vspace{-7mm}
	\centering
	\includegraphics[scale=0.65]{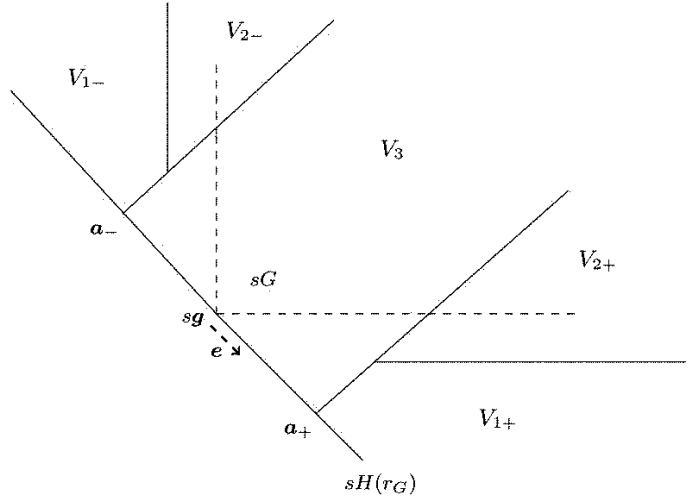}
	\vspace{-5mm}
	\caption{The auxiliary sets $V_{j\pm}$, $j = 1, 2,$ and $V_3$ in the case  $\mathbb{E} \bxi \in -Q^+$.}
	\label{figure:2}
\end{figure}

{\em Case when $\mathbb{E} \bxi \notin -Q^+$} (but  {\rm [\textbf{C}$_3(r_G)$]} is still met, i.e., $\lag \mathbb{E}\bxi, \bv{\zeta} \rag < 0$). Here the  above  simple   construction of the sets $V_{j\pm}$ must be somewhat modified. For definiteness, assume that $\mathbb{E} \xi_2 > 0$, so that    $\mathbb{E} \bxi$ lies  in the interior of the second quadrant, implying that $\lag \mathbb{E} \bxi, \bv{e} \rag<0.$ In that case, all what we have to change in the above definition of the sets $V_{\boldsymbol{\cdot}}$ is to amend how $V_{j+}, j = 1, 2$ are specified ($V_{j-}$ stay the same; in the alternative case, when $\mathbb{E} \xi_1 > 0$, one has to redefine $V_{j-}, j = 1, 2,$ keeping $V_{j+}$ unchanged).

This is done as follows. Introduce the points
\begin{equation*}
\bv{a}'_+ := s\bv{g} + \frac{\mathbb{E} \bxi}{\lag \mathbb{E} \bxi, \bv{e} \rag} M\ln s
\end{equation*}
(which is the intersection of the ray emanating from $s\bv{g}$ in the direction of $-\mathbb{E}\bxi$ and the straight line parallel to $\bv{\zeta}$ and passing through $\bv{a}_+$) and
\begin{align*}
\bv{a}''_+ := \bv{a}_+ + \frac{1}{3}&(\bv{a}'_+ - \bv{a}_+) = s\bv{g} + \bigg(\frac{2}{3}\bv{e} +  \frac{\mathbb{E} \bxi}{3\lag \mathbb{E} \bxi, \bv{e} \rag}\bigg)M\ln s,
\\
\bv{a}_0 &:= s\bv{g} - \bigg( \frac{\mathbb{E} \bxi}{\lag \mathbb{E} \bxi, \bv{e} \rag} - \bv{e} \bigg)\frac{M\ln s}{3}.
\end{align*}
In words, $\bv{a}''_+$ is at one third of the way from $\bv{a}_+$ to $\bv{a}'_+$ going along the direction of~$\bv{\zeta}$, whereas $\bv{a}_0$ is at the same distance from $s\bv{g}$ in the opposite way (see Fig.~\ref{figure:3}).

\begin{figure}[ht]
	\centering
	\includegraphics[scale=0.8]{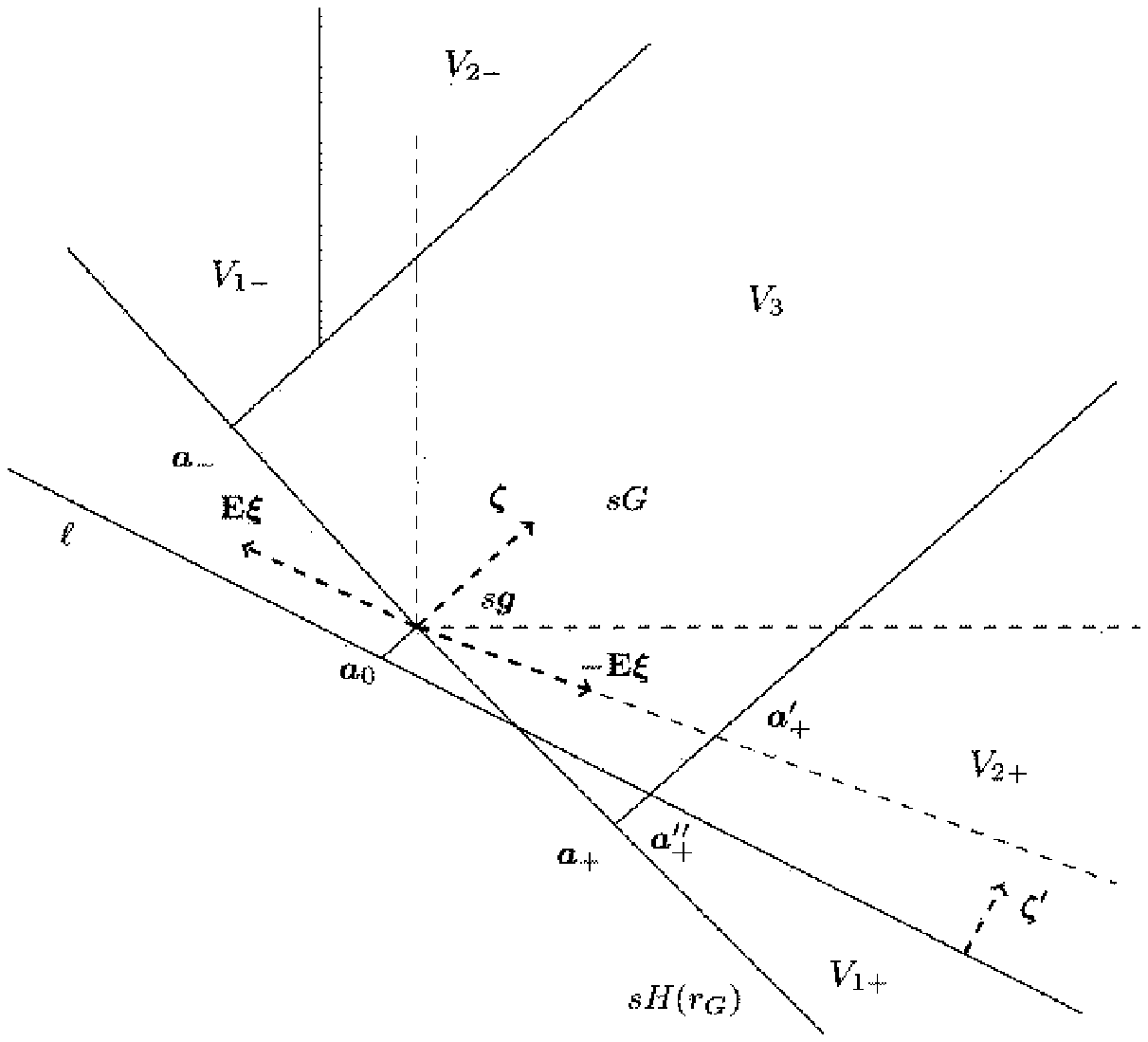}
	\vspace{-5mm}
	\caption{The auxiliary sets $V_{j\pm}$, $j = 1, 2,$ and $V_3$ in the case $\mathbb{E} \bxi \notin -Q^+$.}
	\label{figure:3}
\end{figure}

Now we define $V_{1+}$ as the intersection of $V_+$ with the half-plane lying underneath the straight line $\ell$ going through the points $\bv{a}_0$ and $\bv{a}''_+$:
\begin{equation}
\label{no2}
  V_{1+}:= V_+\cap \bigg\{\bv{v} \in \mathbb{R}^2: \bv{v} = \bv{a}_0 + x\bigg(\frac{1}{3}\bv{e} +  \frac{2\mathbb{E} \bxi}{3\lag \mathbb{E} \bxi, \bv{e} \rag} \bigg) - y\bv{\zeta}, \,\, x\in\mathbb{R}, y \geq 0 \bigg\}.
\end{equation}
All the other sets $V_{\boldsymbol{\cdot}}$ are defined now according to~\eqref{no1}.

For $d\ge 3 $ we   use a general construction of the $V_j$'s (there will only be  three sets here, no need for $V_{j\pm}$) that extends~\eqref{no2}. It is applicable whether $\mathbb{E} \bxi $ lies in  $-Q^+$ or not. We first set $V_3:= \{\bv{v}=s\bv{g}+\bv{u}\in s  \widehat{H} (r_G): \|\bv{u}-\langle \bv{u},\bv{\zeta}\rangle \bv{\zeta}\|\le M\ln s\}$ to be a ``round" half-cylinder in $ s  \widehat{H}(r_G)$ with generatrix parallel to $\bv{\zeta}$ and the base that is the $(d-1)$--dimensional ball that  is a subset of  $ s H(r_G)$,  has its center at $s\bv{g}$ and is of radius~$M\ln s.$ Then we use the cone $C$ described in~\eqref{Cone_C} to define
\[
C_s: = s\bv{g}- \frac{M\ln s}{3\tan \phi}\bv{\zeta}+C,
 \quad V_1:= V_3^c \cap C_s,\quad V_2:= V_3^c \setminus V_1.
\]

Now set $\eta_s := \eta(s\widehat{H}(r_G) )$ and write
\begin{align}
\label{PPP}
\mathbb{P}\big(\eta(sG) < \infty \big) =\sum_{j = 1}^3 \mathbb{P}\big(\eta(sG) < \infty,  \bv{S}(\eta_s) \in V_j \big)=:\sum_{j=1}^3 P_j.
\end{align}
We will show that $P_1$ and $P_2$ are negligibly small compared to the RHS of~\eqref{lastassertion}. After that, we will  use Corollary~\ref{coro1} to demonstrate that, choosing a large enough~$M$, the term $P_3$ can be made arbitrary (relatively) close to the RHS of~\eqref{lastassertion}.

\smallskip

{\it Bounding}  $P_1.$ First we note that in the case $d=2$ one has
\[
P_1=P_{1-}+P_{1+}, \quad P_{1\pm}:= \mathbb{P}\big(\eta(sG) < \infty, \bv{S}(\eta_s) \in V_{1\pm} \big) .
\]
Assume that $\mathbb{E}\bxi \in -Q^+.$ In that case,
\begin{align*}
P_{1+} &:= \int_{V_{1+}}\mathbb{P}\big(\eta(sG) < \infty | \eta_s <\infty, \bv{S}(\eta_s) = \bv{v} \big) \mathbb{P}\big( \eta_s <\infty, \bv{S}(\eta_s) \in d\bv{v}  \big)
\\
& \leq \int_{V_{1+}} \mathbb{P}\Big(\sup_{n \geq 1} S_2(n) \geq 2^{-1} (M\ln s)|e_2| \Big) \mathbb{P}\big(\eta_s < \infty, \bv{S}(\eta_s) \in d\bv{v} \big)
\\
& =  \mathbb{P}\Big(\sup_{n \geq 1} S_2(n) \geq 2^{-1} (M\ln s)|e_2| \Big)\int_{V_{1+}} \mathbb{P}\big(\eta_s < \infty, \bv{S}(\eta_s) \in d\bv{v} \big)
\\
&\leq s^{-c_0M} \mathbb{P}\big(\eta_s <\infty \big), \qquad c_0 := 2^{-1}|e_2|\nu_0 >0,
\end{align*}
where we used the strong Markov property to obtain the first inequality and a bound of the form~\eqref{new50} for the distribution tail of $\sup_{n \geq 1} S_2(n).$ That $|e_2|>0$ is due to condition~[$\textbf{C}_3(r_G)$] (as  it excludes situations where $H(r_G)$ is parallel to any of the coordinate axes). The term $P_{1-}$ is bounded in the same way.

Since $\mathbb{P}\big(\eta_s < \infty \big)~\sim~ce^{-sD(s\widehat{H}(r_G))}$ as $s \to \infty$ by Theorem~7 in \cite{mainpaper} and $ D(\widehat{H}(r_G)) = D(G) $ by Lemma~\ref{DG}, we showed that, for some $0 < c, c_1 < \infty$,
\begin{equation}
\label{star1}
P_1 \leq cs^{-c_1M}e^{ - sD(G)}.
\end{equation}
Choosing $M > 1/(2c_1)$ ($M > (d-1)/(2c_1)$ when $d > 2$) completes the argument.

Now we turn to the case when $\mathbb{E}\bxi \notin Q^+,$ $\mathbb{E}\xi_2 > 0$   and use the alternative construction~\eqref{no2} of $V_{1+}$.  Note that that half-space is separated from $sG$ by a gap of width $cM\ln s$ for some $c > 0$ in the direction orthogonal to~$\ell$.  Further, denote by $\bv{\zeta}'$ a unit vector orthogonal to $\ell$ and such that $\lag \bv{\zeta}, \bv{\zeta}' \rag > 0$ (so that $\bv{\zeta}'$ is pointing in the direction of $sG$). It is easy to verify that, by the above construction, one has
$
\mathbb{E}\lag \bxi, \bv{\zeta}' \rag < 0.
$
This means that we are in the same situation as above, when considering the case $\mathbb{E} \bxi \in -Q^+,$ and can use the same argument to establish that $P_1$ is negligibly small.

The last argument extends in a straightforward way to the case $d \ge 3$ as well:  by construction, in that case the set $V_1$ is ``separated" from $sG$ by a gap of (variable) width $\ge c M\ln s$ for some $c>0. $

\smallskip

{\it  Bounding} $P_2 = \mathbb{P}\big( \eta(sG) < \infty, \bv{S}(\eta_s) \in V_2 \big)$.  We again start with the case $d=2$.
It is clear from our constructions (see Figs.\,2 and \ref{figure:3}) that there exists a $c_2 > 0$ such that $V_2 \subset s_1\widehat{H}(r_G)$ with  $s_1 := s + c_2M\ln s$ (one can take $c_2 := (M\ln s)^{-1}\min_{\sbv{v} \in V_2} \lag \bv{v}, \bv{\zeta} \rag$, where the minimum is attained at the vertex of one of the sets $V_{2\pm}$). Therefore, again using Theorem~7 in~\cite{mainpaper} and our Lemma~\ref{DG}, we have
\begin{align}
P_2 &\leq \mathbb{P}\big(\eta_s < \infty, \bv{S}(\eta_s) \in V_2 \big)
\leq
\mathbb{P}\big(\eta(s_1\widehat{H}(r_G) )  <  \infty \big) \notag
\\
& \sim ce^{-s_1D(\widehat{H}(r_G))}
\label{star2}
= cs^{-c_2MD(G)}e^{ -sD(G)}.
\end{align}
Choosing a large enough $M$, we establish the desired result. There is no change in the argument when $d\ge 3.$

\smallskip

{\it Evaluating} $P_3 = \mathbb{P}\big(\eta(sG) < \infty, \bv{S}(\eta_s) \in V_3 \big).$ Clearly,
\begin{equation}
\label{53a}
P_3  = \sum_{n = 1}^{\infty} P_{3,n},  \qquad P_{3,n} := \mathbb{P}\big(\eta(sG) < \infty, \eta_s = n, \bv{S}(n) \in V_3 \big), \quad n \geq 1.
\end{equation}
First we will compute the sum of the terms $P_{3,n}$ with
\begin{equation*}
n \in N_s := \{n: |n - su_G | \le Ms^{1/2}  \}.
\end{equation*}
In the assertion of Corollary~\ref{coro1}, choose $\gamma(s) := Ms^{1/2},$ where $M = M(s) \to \infty$ slowly enough so that the term $O(\|\bx \|^3/n^2)$ in the exponential in~\eqref{firstresult2} is $o(1)$ for $\|\bx \| \leq \gamma(s)$ (i.e., $M= o(s^{1/6})).$ For a $\Delta > 0$, let $m := (M\ln s) / \Delta$ (we can assume without loss of generality that $m \in \mathbb{N}$). First assume for simplicity that $d=2$ and  set $\bv{t}_k :=  k\Delta \bv{e} $ and $\bv{z}_k := s\bv{g} + \bv{t}_k$, $k = -m, \ldots, m$ (so that $\bv{z}_{-m} = \bv{a}_-$ and $\bv{z}_m = \bv{a}_+).$  Recalling that $r = 1/u$ and $r_G = 1/u_G$, in view of Corollary~\ref{coro1} with $\bv{x} = \bv{x}_k := \bv{z}_k - n\bbe(1/u) \equiv \bv{t}_k + s\bv{g} - n\bbe(1/u)$,  we have
\begin{align}
\label{p3n_sum}
P_{3,n} &= \mathbb{P}\big(\eta(sG) < \infty, \eta_s = n, \bv{S}(n) \in V_3  \big)  \notag
\\
&=  \sum_{k=-m}^{m-1}  \mathbb{P}\big(\eta(sG) < \infty, \eta_s = n, \bv{S}(n) \in W(\Delta^*[\bv{z}_k))  \big) \notag \\
& = (1+o(1))\Xi(s,n)\!\!\! \sum_{k = -m}^{m-1}\!\! e^{-\frac{1}{2n}\sbv{x}_k \Lambda''(\sbbe(1/u)) \sbv{x}_k^T}
 \notag
 \\
 &\hspace{32 mm} \times E(1/u, \bv{t}_k, W(\Delta^*[\bv{0}))) + o\big(\Xi(s,n)\big),
\end{align}
where the remainder term $o\big(\Xi(s,n)\big)$ appears as the result of summing up the terms $R$ in~\eqref{firstresult2}, as one can easily verify that $\int_{H_0(r_G)} e^{-c_1\|\sbv{w} \|} d\mu(\bv{w}) < \infty.$

Next observe that
\begin{equation*}
E(1/u, \bv{t}_k, W(\Delta^*[\bv{0})))  = \int_{\Delta^*[\sbv{t}_k)} \rho_u(\bv{t})d\mu(\bv{t}),
\end{equation*}
where we put, for  $\bv{t} \in H_0(1/u_G),$
\begin{equation*}
\rho_u(\bv{t}) := \int_{0}^{\infty}e^{-\lag \sbla(\sbbe(1/u)), \sbv{t} -\sbv{t}_k+ y\sbv{\zeta} \rag} q_{\sbbe(1/u)}(\bv{t}-\bv{t}_k + y\bv{\zeta})  p(\bv{t} + y\bv{\zeta}) dy.
\end{equation*}
Note that since $e^{- \lag \sbla(\sbbe(1/u)), \sbv{t}-\sbv{t}_k+ y\sbv{\zeta} \rag} = e^{-\lag \sbla(\sbbe(1/u)),  y\sbv{\zeta} \rag}$ and  $ q_{\sbbe(1/u)}(\bv{t}-\bv{t}_k + y\bv{\zeta})  =  q_{\sbbe(1/u)}( y\bv{\zeta}) $ for $\bv{t} \in H_0(1/u_G)$, one actually has
\begin{equation*}
\rho_u(\bv{t}) = \int_{0}^{\infty}e^{-\lag \sbla(\sbbe(1/u)), y\sbv{\zeta} \rag} q_{\sbbe(1/u)}( y\bv{\zeta})  p(\bv{t} + y\bv{\zeta}) dy.
\end{equation*}

Recalling our notation \eqref{chi}, the sum on the RHS of~\eqref{p3n_sum} can be expressed as
\begin{align*}
\sum_{k = -m}^{m-1} e^{-\frac{1}{2n}(\sbv{t}_k  - \sbv{\chi}) \Lambda''(\sbbe(1/u))(\sbv{t}_k  - \sbv{\chi})^T} \int_{\Delta^*[\sbv{t}_k)} \rho_u(\bv{t})d\mu(\bv{t}).
\end{align*}
Putting $f(\bv{z}) := \exp\{-\frac{1}{2n}\bv{z} \Lambda''(\bbe(1/u))\bv{z}^T \} $,  one can easily verify that
\begin{equation}
\label{fdiff}
\frac{f(\bv{z} + \Delta_1\bv{e})}{f(\bv{z})} = 1+ o(1)
\end{equation}
uniformly in $n \in N_s$, $\Delta_1 \in (0, \Delta]$ and $\| \bv{z}\| \leq cMs^{1/2}$, $c > 0$.

Therefore, letting $\Delta \to 0$ sufficiently slowly,  we can replace the above sum with the integral over the set $\Delta^*_0 [\bv{a}_-)$ with $\Delta^*_0 := 2m\Delta \equiv 2M\ln s$ to obtain
\begin{equation}
\label{p33}
P_{3,n} =  (1+o(1)) \Xi(s,n)\int_{\Delta^*_0[\sbv{a}_-)} e^{-\frac{1}{2n}(\sbv{t}  - \sbv{\chi}) \Lambda''(\sbbe(1/u)) (\sbv{t}  - \sbv{\chi})^T}\rho_u(\bv{t})d\mu(\bv{t}) + o\big(\Xi(s,n)\big).
\end{equation}
Recalling that $\bv{a}_- = -(M\ln s) \bv{e},$ we have from Lemma~\ref{nbconv} (with $\gamma(s) = Ms^{1/2}$) that
\begin{align*}
\exp &\Big\{-\frac{1}{2n}(\bv{t}-\bv{\chi})\Lambda''(\bbe(1/u))(\bv{t}-\bv{\chi})^T \Big\}
\\
&= \exp \Big\{ -\frac{1}{2n}\bv{\chi}\Lambda''(\bbe(1/u))\bv{\chi}^T +\frac{1}{n}\bv{t}\Lambda''(\bbe(1/u))\bv{\chi}^T - \frac{1}{2n}\bv{t}\Lambda''(\bbe(1/u))\bv{t}^T \Big\}
\\
&= (1+o(1))\exp \Big\{  -\frac{1}{2n}\bv{\chi}\Lambda''(\bbe(1/u))\bv{\chi}^T\Big\}
\end{align*}
uniformly in  $\bv{t} \in \Delta_0^*[\bv{a}_-)$ and $n \in N_s$.
Hence  it follows from~\eqref{p33} that
\begin{equation*}
P_{3,n} = (1+o(1))\Xi(s,n)e^{-\frac{1}{2n}\sbv{\chi}\Lambda''(\sbbe(1/u))\sbv{\chi}^T }  \int_{\Delta^*_0[\sbv{a}_-)} \rho_u(\bv{t})d\mu(\bv{t}) + o\big(\Xi(s,n)\big).
\end{equation*}
Note that $\int_{\Delta^*_0[\sbv{a}_-)} \rho_u(\bv{t})d\mu(\bv{t}) = E\big(1/u, \bv{0}, W(\Delta^*[\bv{a}_-))\big)$ and, as $M \to \infty$, one has $ E(1/u, \bv{0}, W(\Delta^*[\bv{a}_-))) \to E(1/u, \bv{0}, \widehat{H}_0(1/u_G))$, so that
\begin{align}
\label{newp3}
 P_{3,n}  = (1+o(1))\Xi(s,n)e^{-\frac{1}{2n}\sbv{\chi}\Lambda''(\sbbe(1/u))\sbv{\chi}^T } E(1/u, \bv{0}, \widehat{H}_0(1/u_G)) +o\big(\Xi(s,n)\big).
\end{align}
Representation \eqref{newp3} holds in the case   $d\ge 3  $ as well. This is shown using the same argument as above, the only difference being that, instead of partitioning the straight line segment with end points $\bv{a}_-$ and $\bv{a}_+$ into small subintervals  $\Delta^*[\bv{z}_k),$ we partition the base of the half-cylinder $V_3$ into small cubes (showing that the ``boundary effects" arising due to the ``imperfection" of such a partition of that ball will be negligible).

Recalling the representation $ \bv{\chi}  = (n - s u_G) \bv{\kappa} + O(s^{-1}\gamma^2(s))$ from Lemma~\ref{nbconv} and setting
\begin{equation}
\label{no4}
a(u) := \bv{\kappa}\Lambda''(\bbe(1/u))\bv{\kappa}^T,
\end{equation}
we see that, for $|n-s u_G|\le \gamma (s),$ one has
\begin{align*}
\exp \Big\{-\frac{1}{2n}\bv{\chi}\Lambda''(\bbe(1/u))\bv{\chi}^T \Big\}
&=\exp \Big\{ -\frac{1}{2n}\big[ a(u)( n - su_G)^2 + O(s^{-1}\gamma^3(s) )\big]   \Big\}
\\
&= \exp \Big\{- a(u)s\frac{(u- u_G)^2}{2u} + O\big( s^{-2} \gamma^3(s)  \big) \Big\}
\\
&= \exp \Big\{- a(u)s\frac{ (u - u_G)^2}{2u}\Big\}(1+o(1))
\\
&= \exp \Big\{- a(u_G)s\frac{(u - u_G)^2}{2u} \Big\}(1+o(1))
 \end{align*}
since $\gamma (s) =Ms^{1/2}, $ $M =o(s^{1/6}),$ and $|u - u_G| \leq Ms^{-1/2}$ for $n \in N_s$ and the function $a(u)$ is continuous.

Recalling~\eqref{doubleinf1},~\eqref{XIXI} and that $n = su$, one has
 \begin{equation*}
 \Xi(s,n) = \frac{e^{-sD_u(\widehat{H}(1/u_G))}}{(2\pi s)^{d/2}u^{d/2}\sigma(\bbe(1/u))}.
  \end{equation*}

We conclude that the first term on the RHS of~\eqref{newp3}, after the substitution $n = su,$ takes (up to the factor $1+o(1)$) the following form:
\begin{multline*}
\pi_s(u) := \frac{1}{(2\pi s)^{d/2}u^{d/2}\sigma(\bbe(1/u))}
\\
\times \exp \Big\{ -sD_u(\widehat{H}(1/u_G)) -  a(u_G)s\frac{(u-u_G)^2}{2u} \Big\}E(1/u, \bv{0}, \widehat{H}_0(1/u_G)),
\end{multline*}
and so in this part of the proof we are aiming at computing the sum
\begin{equation}
\label{no3}
\sum_{n \in N_s} P_{3, n} = (1+o(1))\sum_{n \in N_s}\pi_s(n/s) + \sum_{n \in N_s}o(\Xi(s,n)).
\end{equation}
To replace the first sum on the RHS of~\eqref{no3} by the respective integral w.r.t.~$du$, we note that, for $0 \leq \theta < 1$ and $u \in [u_G - Ms^{-1/2}, u_G + Ms^{-1/2}] =: I_s,$ one has
\begin{equation*}
\frac{\pi_s(u + \theta/s)}{\pi_s(u)} = 1+ o(1).
\end{equation*}
This can be verified by an elementary calculation, using the continuity of $\bbe(1/u)$ and $E(1/u, \bv{0}, \widehat{H}_0(1/u_G))$  in $u$, and also the fact that, by the mean value theorem,
\begin{equation*}
D_{u + \theta/s}(\widehat{H}(1/u_G)) = D_{u}(\widehat{H}(1/u_G)) + D'_{u}(\widehat{H}(1/u_G))\big|_{u+\theta^*/s} \theta/s
\end{equation*}
  for some $\theta^* \in (0, \theta),$ where $D'_{u}(\widehat{H}(1/u_G))\big|_{u+\theta^*/s} = o(1)$ uniformly in $u \in I_s$ since $D'_{u}(\widehat{H}(1/u_G))\big|_{u=u_G} = 0$ (cf.\ the proof of Lemma~\ref{DG}). Therefore, the first sum on the RHS of~\eqref{no3} equals
\begin{equation}
\label{p'3}
\widetilde P_3 := \frac{(1+o(1))\widetilde{E}_{u_G}}{(2\pi)^{d/2}s^{d/2-1}} \int_{I_s}  \exp \Big\{ -sD_u(\widehat{H}(1/u_G)) -  a(u_G)s\frac{(u-u_G)^2}{2u} \Big\}du,
\end{equation}
where we used the fact that
\begin{equation*}
  \widetilde{E}_u := \frac{E(1/u, \bv{0}, \widehat{H}_0(1/u_G))}{u^{d/2}\sigma(\bbe(1/u))} = (1+o(1))\widetilde{E}_{u_G} \quad {\rm for }  \quad u \in I_s.
\end{equation*}

To be able to apply now the Laplace method for evaluating the integral on the RHS of~\eqref{p'3}, we will need the following lemma.

\begin{lemma}
\label{DHconvex}
There exists a $\delta > 0$ such that the function $D_u(\widehat{H}(r_G))$ is convex on the interval $(u_G -\delta, u_G + \delta).$
\end{lemma}

\begin{proof}
First note that, in view of {\rm [\textbf{C}$_3(r_G)$]}, there is a  $\delta > 0 $ such that  $\bbe(1/u)$ is well-defined for $u \in (u_G - \delta, u_G + \delta).$ That the function $D_u(\widehat{H}(r_G)) = u\Lambda(\widehat{H}(r_G)/u) \equiv u\Lambda(\bbe(1/u))$  is convex in $u$ on that interval  means that, for $u_1, u_2 \in (u_G -\delta, u_G + \delta)$, $a \in (0, 1)$ and $u_0 := au_1 + (1-a)u_2$, one has
\begin{equation}
\label{convexh}
u_0\Lambda(\bbe(1/u_0)) \leq au_1\Lambda(\bbe(1/u_1)) + (1-a)u_2\Lambda(\bbe(1/u_2)).
\end{equation}
Recall that $\bbe(1/u)$ is the MPP of the set $\frac{1}{u}\widehat{H}(r_G)$ and, as $\Lambda$ is convex, that point is located on the boundary~$\frac{1}{u}H(r_G)$ of that set by~\eqref{aldB}. By [{\bf D}$_4$] (setting $\bv{v}_k := u_k \bbe(1/u_k)$ in \eqref{convexineq}), letting
$
\bbe_0 :=  \frac{au_1}{u_0}\bbe(1/u_1) + \frac{(1-a)u_2}{u_0}\bbe(1/u_2),
$
one has
\begin{equation}
\label{59a}
u_0\Lambda ( \bbe_0 ) \leq au_1 \Lambda(\bbe(1/u_1)) + (1-a)u_2\Lambda(\bbe(1/u_2)).
\end{equation}
On the other hand, as $\bbe(1/u) \in \frac{1}{u}H(r_G)$, one also has
\begin{equation*}
\frac{au_1}{u_0}\bbe(1/u_1) \in \frac{a}{u_0}H(r_G) \quad {\rm and} \quad \frac{(1-a)u_2}{u_0}\bbe(1/u_2) \in \frac{1-a}{u_0}H(r_G).
\end{equation*}
Hence we conclude that $\bbe_0 \in \frac{1}{u_0}H(r_G).$ However, $\bbe(1/u_0)$ is the MPP of the ``upper'' half-space $\frac{1}{u_0}\widehat{H}(r_G),$ and therefore $\Lambda(\bbe(1/u_0)) \leq \Lambda(\bbe_0).$ Together with~\eqref{59a} this proves~\eqref{convexh}.
\end{proof}

Now it follows that the function in the exponential in~\eqref{p'3} is concave and continuously differentiable in a neighborhood  of the point $u = u_G$ at which it attains its maximum value equal to $-sD_u(\widehat{H}(1/u_G)) = -sD(G)$ (by Lemma~\ref{DG}). Furthermore, there exist (see\ (28) in \cite{mainpaper})
\begin{equation*}
\sigma^2_D := \frac{d^2}{du^2}D_u(\widehat{H}(1/u_G))\Big|_{u=u_G} > 0 \quad {\rm and }
\quad \frac{d^2}{du^2} \bigg(\frac{(u-u_G)^2}{u} \bigg) \bigg|_{u= u_G} = \frac{2}{u_G}.
\end{equation*}
By the routine use of the Laplace method (see\ e.g.\ Section 2.4 in \cite{erd_asy}), recalling that we let $M = M(s) \to \infty$, we obtain that the integral in~\eqref{p'3} equals
\begin{equation*}
(1+o(1))e^{-sD(G)}\sqrt{\frac{2\pi}{s(\sigma^2_D+a(u_G)/u_G)}}.
\end{equation*}
Therefore, letting $\sigma^*_D := \sqrt{\sigma^2_D +  a(u_G)u_G^{-1}}$, we have
\begin{equation}
\label{no5}
\widetilde P_3 = \frac{(1+o(1))E(1/u_G, \bv{0}, \widehat{H}(1/u_G))}{(2\pi)^{(d-1)/2}u_G^{d/2}\sigma^*_D\sigma(\bal(1/u_G))} \cdot \frac{e^{-sD(G)}}{s^{(d-1)/2}}.
\end{equation}
It remains to compute the sum of the second terms $o\big(\Xi(s,n)\big)$ in~\eqref{no3} over $n \in N_s$. Applying the Laplace method in the same way as when evaluating $\widetilde P_3$, we find that
\begin{equation}
\label{oxi}
\sum_{n \geq 1} \Xi(s,n) = O(\widetilde P_3).
\end{equation}
So the above-mentioned sum of the remainders is $o(\widetilde P_3). $ We conclude  from~\eqref{no3} that
\begin{equation}
\label{no6}
\sum_{n \in N_s}P_{3,n} = (1+o(1)) \widetilde P_3.
\end{equation}

Next we will bound the sum $\sum_{n \notin N_s}P_{3,n}$. For a fixed $\gamma \in \mathscr{G}$ (to be chosen later, after~\eqref{p331}; we will need a function growing faster than $Ms^{1/2}$), let
\begin{align*}
N^*_s  := \{n \in \mathbb{N}: Ms^{1/2} < |n - su_G| \leq \gamma(s) \},
\quad
N^{**}_s  := \{n \in \mathbb{N}:  |n - su_G| > \gamma(s) \},
\end{align*}
and show that the sums of $P_{3,n}$ over $n \in N^*_s$ and $n \in N^{**}_s$ are both $o(\widetilde P_3).$ These sums will have to be bounded in different ways,  the sum over $N^{**}_s$  being easier to handle.

Consider the sum over~$n\in N^{*}_s$. It will again be easier   to first explain the proof in the case $d=2$; it is extended to the general case using the same argument as presented after representation~\eqref{newp3}.  By Corollary~\ref{coro1},  for $n \in N^*_s$ expression \eqref{p3n_sum} becomes
\begin{multline*}
P_{3,n} = (1+o(1))\Xi(s,n)\sum_{k = -m}^{m-1} \Big[\exp \Big\{-\frac{1}{2n}\bv{x}_k \Lambda''(\bbe(1/u)) \bv{x}_k^T + O(\| \bx_k\|^3n^{-2})\Big\}
\\
\times E\big(1/u, \bv{t}_k, W(\Delta^*[\bv{0}))\big)\Big] + o\big(\Xi(s,n)\big),
\end{multline*}
where the remainder term $o\big(\Xi(s,n) \big)$ is the same as the one in~\eqref{p3n_sum}. It will turn out that, for $n \in N^*_s,$ the values of $\bx_k$ will be large enough to ensure the desired result due to the quadratic term in the exponential in the sum.

Recall that $\bx_k =  \bv{t}_k + s\bv{g} - n\bbe(1/u)$. Since $\|s\bv{g} - n\bbe(1/u) \| < c\gamma(s)$ for $n \in N^*_s$ by Lemma~\ref{nbconv} and $\|\bv{t}_k \| \le  M \ln s$,  $k = -m, \ldots, m,$ one has $\|\bx_k \| < c_1\gamma(s)$, $k = -m, \ldots, m$. It is not hard to verify that  relation~\eqref{fdiff} holds true for $\|\bv{z}\| < \gamma(s)$ as well. Therefore, setting
\[
\Upsilon (n,s,\bv{t}):=\frac{1}{2n}(\bv{t}  - \bv{\chi}) \Lambda''(\bbe(1/u)) (\bv{t}  - \bv{\chi})^T
\]
and
following steps similar to the ones used to obtain~\eqref{p33}, one has, for $n \in N^*_s,$
\begin{multline}
\label{p331}
P_{3,n}   = (1+o(1)) \Xi(s,n) \int_{\Delta^*_0[\sbv{a}_-)}
  \!\! \!\!\!
 e^{- \Upsilon (n,s,\sbv{t})  + O(\| \sbv{\chi}\|^3 n^{-2})}
  \rho_u(\bv{t})d\mu(\bv{t}) + o\big(\Xi(s,n)\big).
\end{multline}
Now choose $\gamma (s):= s^{5/8},$ $M:=s^{1/10}$ (thus ensuring that $\gamma (s) \gg  Ms^{1/2}$ and $M=o(s^{1/6})$, as required) and consider the first factor in the integrand. By Lemma~\ref{nbconv}, one has   $\| \bv{\chi}\|^3 n^{-2} = O(\gamma^3(s)s^{-2})= o(1) $ for $n\in N^*_s$. Further, as $\|\bv{t}\|\le M\ln S,$ due to the same lemma, using a computation similar to the one following~\eqref{no4}, we have for $n$ from the same range that
\begin{align*}
\Upsilon (n,s,\bv{t}) & =\frac1{2n} (n-su_G)^2 a(u) + O\big(s^{-2}\gamma^3 (s)+s^{-1}\gamma (s)M\ln s\big)
\\
&\ge \frac{a(u)}{2u} M^2 +o(1) = \frac{a(u_G)}{2u_G}M^2 (1+o(1)) \ge c_0 M^2
\end{align*}
for some $c_0>0,$ as $a(u)/u \to a(u_G)/u_G > 0.$

Now recalling that $\bv{a}_- = -(M\ln s)\bv{e}$, $\Delta_0 = 2M\ln s$ and  $M \to \infty$, we see that the expression on the RHS of~\eqref{p331} does not exceed
\begin{align*}
(1+o(1))\Xi(s,n)&e^{-c_0 M^2} \int_{\Delta^*_0[\sbv{a}_-)} \rho_u(\bv{t})d\mu(\bv{t}) +  o\big(\Xi(s,n)\big)
 \\
&= \big(E(\bbe(1/u), \bv{0}, \widehat{H}_0(1/u_G)) + 1\big) o\big(\Xi(s,n)\big) = o\big(\Xi(s,n)\big)
\end{align*}
as $E(1/u, \bv{0}, \widehat{H}_0(1/u_G)) < \infty.$ Therefore it follows from~\eqref{oxi} that
\begin{align}
\label{2ndsum}
\sum_{n \in N^*_s} P_{3,n} = o\Big(\sum_{n \in N^*_s} \Xi(s,n) \Big) = o\Big(\sum_{n \geq 1} \Xi(s,n) \Big) = o( \widetilde P_3).
\end{align}
This bound is obtained in the case $d\ge 3$ in exactly the same way, using the same change in the argument  as described in the paragraph following~\eqref{newp3}.

It remains to evaluate the term $\sum_{n \in N^{**}_s} P_{3,n}.$ From~\eqref{53a} and Chebyshev's exponential inequality, one has
\begin{equation*}
P_{3,n} \leq \mathbb{P}\big( \bv{S}(n) \in s\widehat{H}(1/u_G) \big) \leq e^{-n\Lambda(\sbbe(1/u))} = e^{-sD_u(\widehat{H}(1/u_G))}.
\end{equation*}
Recall that $D_u(\widehat{H}(r_G))$ is convex in a neighborhood of $u_G$ and attains its minimum at $u_G$, with  $\frac{d}{du}D_u(\widehat{H}(r_G))\big|_{u=u_G} = 0$ and $ \frac{d^2}{du^2} D_u(\widehat{H}(r_G))\big|_{u = u_G} = \sigma^2_D > 0$.  Setting $n_1 := |n - u_Gs|$, for some $\delta > 0$ we see that, for our chosen  $\gamma (s)=s^{5/8}, $  one has, for some $c_k \in (0, \infty),$ $1\le k \le 5,$ the bounds
\begin{align*}
\sum_{n \in N^{**}_s}  & e^{-sD_{n/s}(\widehat{H}(1/u_G))}
 \\
 &\leq 2e^{-sD(G)}\Big(\sum_{\gamma(s) < n_1 \leq \delta s}e^{-c_1n_1^2/s} + \sum_{ n_1 > \delta s} e^{-c_3s\delta^2 - c_2s(n_1-\delta s)} \Big)
\\
&\leq  2e^{-sD(G)} \Big(\frac{c_4s}{\gamma(s)}e^{-c_1\gamma^2(s)/s  } + c_5e^{-c_3s\delta^2}\Big)   =o(\widetilde P_3).
\end{align*}
Together with~\eqref{53a},~\eqref{no5},~\eqref{no6} and~\eqref{2ndsum}, that leads to
\begin{equation*}
P_3 =  A s^{-(d-1)/2}e^{-sD(G)}(1+o(1)),
\end{equation*}
 where
 \begin{equation}
\label{AAA}
 A:= \frac{E(1/u_G, \bv{0}, \widehat{H}(1/u_G))}{(2\pi)^{(d-1)/2}u_G^{d/2}\sigma^*_D\sigma(\bal(1/u_G))}.
\end{equation}
Together with  \eqref{star1} and \eqref{star2}, this completes the proof of Theorem~\ref{finalthm}.
\end{proof}

\section{A Numerical Example}\label{sect_4}

To illustrate our main result, we will present the outcome  of a simulation study where we used an importance sampling algorithm to get Monte Carlo estimates for  $\mathbb{P}(\eta (sG)<\infty)$   for a range of $s$ values in the case of a bivariate RW with a normal jump distribution.

The estimate is based on the change-of-measure representation
\[
\mathbb{P}(\eta(sG)<\infty)
 =\mathbb{E}_{\sbv{\lambda}} e^{-\langle\sbv{\lambda}, \sbv{S}(\eta(sG))\rangle} ,
\]
where $\mathbb{E}_{\sbv{\lambda}}$ is the expectation w.r.t.\ the probability measure $\mathbb{P}_{\sbv{\lambda}},$ under which the $\bv{\xi}_i$'s are i.i.d.\ random vectors with distribution $F_{\sbv{\lambda}},$ and $\bv{\lambda}$ is chosen so that
\begin{equation}
\label{lapr}
\psi (\bv{\lambda})=1,\qquad \mathbb{P}_{\sbv{\lambda}}(\eta(sG)<\infty)=1.
\end{equation}

We took $F$ to be the bivariate normal distribution $N(\bv{\mu},\Sigma)$ with a non-degenerate~$\Sigma$, in which case  clearly $ \psi (\bv{\lambda}) = \exp\{\bv{\mu} \bv{\lambda}^\top
+ \frac12  \bv{\lambda} \Sigma  \bv{\lambda}^\top\},
$
and  the first relation in~\eqref{lapr} is satisfied on an ellipse passing through the origin. Further, one can easily show that here $\Lambda (\bal)=\frac12  (\bal- \bv{\mu}) \Sigma^{-1}  ( \bal- \bv{\mu})^\top$ and, given that condition  [\textbf{C}$_3(r_G)$] is satisfied  (so that, in particular, $D(G)=D(\bv{g})$), one has
\[
D(G) = \frac1{2t(\bv{g})}( t(\bv{g})\bv{g}- \bv{\mu}) \Sigma^{-1}  ( t(\bv{g})\bv{g}- \bv{\mu})^\top,
\]
where $t(\bv{g})$ solves the equation $\frac{d}{dt}(\Lambda (t\bv{g})/t)=0.$

For our numerical example, we chose
\[
\bv{\mu}:= (-0.5, -0.3),
\qquad
\Sigma :=\biggl(
\begin{array}{cc}
1 & 0.4\sqrt{ 0.8}\\
0.4\sqrt{ 0.8} & 0.8
\end{array}
\biggr),
\qquad
\bv{g}:= (1.5, 2).
\]
It is easy to verify that all the conditions [\textbf{C}$_1$], [\textbf{C}$_2$]    and [\textbf{C}$_3(r_G)$] are met in this case. Next we had to  choose a $\bv{\lambda}$   that would satisfy~\eqref{lapr};  we took $\bv{\lambda}^*:=(0.5331315, 0.7108420)$ (in which case $\psi(\bv{\lambda}^*)-1\approx 2.6\times 10^{-8}$). A routine computation yields $D(G)\approx 2.22939.$

We simulated $5\times 10^4$ trajectories of $\bv{S}^{(\sbv{\alpha})}$ with $
\bal:=\bal( \bv{\lambda}^*)=(0.2874500,0.4594125).
$
For each trajectory, we simulated the first $350$ steps (that was always enough to hit~$sG$ with~$s=15$ in our experiment),  testing at each step the condition that the RW $\bv{S}^{(\sbv{\alpha})}$ hits $sG$ for each  $s=7+0.02k$, $k=0,1,\ldots, 400.$ Taking then the sample means of  $ e^{-\langle\sbv{\lambda^*} ,\sbv{S}(\eta(sG))\rangle}$ yielded simultaneous estimates for $\mathbb{P}(\eta(sG)<\infty) $ for all $s$-values from the above grid.
\begin{figure}[ht]
 	\vspace{-2mm}
		\centering
	\includegraphics[scale=0.6]{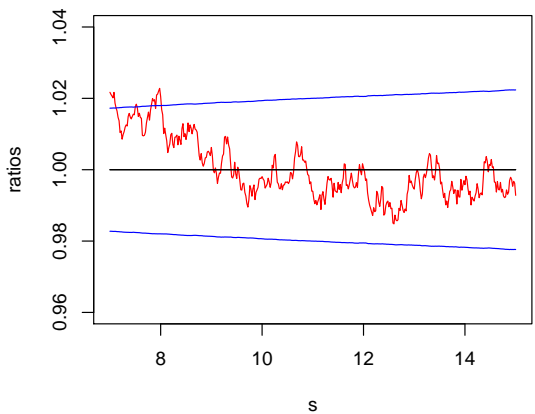}
	\vspace{-5mm}
	\caption{The ratio of the main term in the theoretical  asymptotics~\eqref{lastassertion} for $\mathbb{P}(\eta (sG)<\infty)$  for  the normal RW to the Monte Carlo estimates, together with the ends of the 99\% confidence intervals thereof, for  $s\in [7,15].$}
	\label{ratio_plot}
\end{figure}

Fig.~\ref{ratio_plot} presents the ratio of the main term $As^{-1/2}e^{-sD(G)}$ on the RHS of~\eqref{lastassertion} to the Monte Carlo estimates for $s\in [7,15],$ together with the 99\%\ confidence intervals (obtained as discussed on p.~463 in~\cite{ASM}). As computing the theoretical value of~$A$ is somewhat cumbersome,  for the purposes of the present illustration  we used  the value of~$A$ obtained by fitting the simulation data (which yielded $A\approx 0.3396$), concentrating on verifying the functional form of~\eqref{lastassertion}.  Fitting that formula to the values of the Monte Carlo estimates yielded   $D(G)\approx 2.22954$ (so that the relative error for the second rate function is less than~$10^{-4}$).
The plot shows remarkable stability for the ratio, thus confirming the validity of our main result.


\medskip

{\bf Acknowledgements.}
This research was funded partially   by the Australian Government through the Australian Research Council's Discovery Projects funding scheme (project DP150102758).  Y.~Pan was also supported by the Australian Postgraduate Award and   the School of Mathematics and Statistics, The University of Melbourne. The authors are grateful to the anonymous referee for comments that helped to improve the exposition of the paper.


%
%
%
%

\end{document}